\documentclass[12pt,leqno]{amsart}
\usepackage{amsmath,color, multirow}
\usepackage[round]{natbib}
\usepackage{graphicx}
\usepackage{tikz}
\usepackage{pgfplots}
\newcommand{\R}{{\mathbb R}}

\DeclareMathOperator{\Cov}{Cov}

\DeclareMathOperator{\Ex}{E}
\DeclareMathOperator{\Px}{P}
\newcommand{\ignore}[1]{}

\newcommand{\N}{\mathcal{N}}
\newcommand{\RR}{\mathbb R}
\newcommand{\NN}{\mathbb N}
\newcommand{\X}{\boldsymbol{X}}
\newcommand{\Y}{\boldsymbol{Y}}
\newcommand{\U}{\boldsymbol{U}}
\newcommand{\XI}{\boldsymbol{\xi}}

\newcommand{\bbr}{\mathbb{R}}

\newcommand{\bbn}{\mathbb{N}}

\newcommand{\abs}[1]{\left| #1 \right|}

\newtheorem{lemma}{Lemma}[section]

\newtheorem{theorem}[lemma]{Theorem}
\newtheorem{proposition}[lemma]{Proposition}

\theoremstyle{definition}
\newtheorem{definition}[lemma]{Definition}
\newtheorem{example}[lemma]{Example}

\newtheorem{remark}[lemma]{Remark}

\begin{document}
\bibliographystyle{plainnat}
\title[Ordinal Patterns for LRD Time Series ]
{Ordinal Patterns in Long-Range Dependent Time Series}
\author[A. Betken]{Annika Betken}
\author[J. Buchsteiner]{Jannis Buchsteiner}
\author[H. Dehling]{Herold Dehling}
\author[I. M\"unker]{Ines M\"unker}
\author[A. Schnurr]{Alexander Schnurr}
\author[J. Woerner]{Jeannette H.C. Woerner}
\today

\address{Fakult\"at f\"ur Mathematik, Ruhr-Universit\"at Bo\-chum,  44780 Bochum, Germany}
\email{annika.betken@rub.de}
\address{Fakult\"at f\"ur Mathematik, Ruhr-Universit\"at Bo\-chum,  44780 Bochum, Germany}
\email{jannis.buchsteiner@rub.de}

\address{Fakult\"at f\"ur Mathematik, Ruhr-Universit\"at Bo\-chum,  44780 Bochum, Germany}
\email{herold.dehling@rub.de}
\address{Department of Mathematics, University of Siegen, 57068 Siegen, Germany}
\email{muenker@mathematik.uni-siegen.de}
\address{Department of Mathematics, University of Siegen, 57068 Siegen, Germany}
\email{schnurr@mathematik.uni-siegen.de}
\address{Department of Mathematics, TU Dortmund, 44221 Dortmund}
\email{jwoerner@mathematik.uni-dortmund.de}


\thanks{This research was supported in part by the German Research Foundation (DFG) through  Collaborative Research Center SFB 823  {\em Statistical Modelling of Nonlinear Dynamic Processes}, Research Training Group RTG 2131 {\em High-dimensional Phenomena in Probability - Fluctuations and Discontinuity} and the project \emph{ Ordinal-Pattern-Dependence: Grenzwerts\"atze und Strukturbr\"uche im langzeitabh\"angigen Fall mit Anwendungen in Hydrologie, Medizin und Finanzmathematik} (SCHN 1231/3-1).}
\begin{abstract}
We analyze the ordinal structure of long-range dependent time series. 
To this end, we use so called ordinal patterns which describe the relative position of consecutive data points.
We provide two estimators for the probabilities of ordinal patterns and prove limit theorems in different settings, namely stationarity and (less restrictive) stationary increments. In the second setting, we encounter a Rosenblatt distribution in the limit. We prove more general limit theorems for functions with Hermite rank 1 and 2. We derive the limit distribution for an estimation of the Hurst parameter $H$ if it is higher than 3/4. Thus, our theorems complement results for lower values of $H$ which can be found in the literature. Finally, we provide some simulations that illustrate our theoretical results. 
\end{abstract}
\maketitle

\section{Introduction}\noindent
Originally, ordinal patterns have been introduced to analyze long and noisy time series. They have proved to be useful in various contexts such as sunspot numbers (\cite{bandt:shiha:2007}), EEG data (\cite{keller:maksymenko:stolz:2015}), speech signals (\cite{bandt:2005}) and chaotic maps which appear in the theory of dynamical systems (\cite{bandt:pompe:2002}).
Further applications include the approximation of the Kolmogorov-Sinai entropy (\cite{sinn:2012}). Recently, ordinal patterns have been used to detect and to model dependence structures between time series; see \cite{schnurr:2014}. Limit theorems for the parameters under consideration have been proved in the short-range dependent setting in \cite{schnurr:dehling:2017}.\par
In the present paper we will investigate ordinal patterns in the long-range dependent setting. To the best of our knowledge \cite{sinn:keller:2011} is the only article which explicitly deals with the interplay between ordinal patterns and the Hurst parameter $H$. The authors estimate this parameter of a fractional Brownian motion restricting their considerations to $H<\frac{3}{4}$. 
An overview and a comparison of various other techniques for estimating the Hurst parameter is given in \cite{taqqu:et_al:1995} and \cite{rea:et_al:2009}. None of the therein considered methods
 requires a restriction on the range of admissible values for $H$.
Nonetheless, graphical methods that are used to estimate the Hurst parameter such as  the aggregated variance method or  the R/S method (\cite{mandelbrot:wallis:1969}, \cite{mandelbrot:1975} and \cite{mandelbrot:taqqu:1979}) are known to be biased. 
Estimators operating in the frequency domain of time series, such as the Whittle estimator, 
which are usually based on an estimation of the spectral density by the periodogram,
often make parametric assumptions on the spectral density of the data-generating process.
Semiparametric alternatives such as the GPH estimator (\cite{geweke:porter-hudak:1983})
and
  the local Whittle estimator (\cite{kuensch:1987}, \cite{robinson:1995})   require the choice of  
a bandwidth parameter denoting the number   of  Fourier frequencies incorporated in the estimation of the spectral density by the periodogram. 
The choice of this tuning parameter is crucial to the performance of semiparametric estimates, but difficult to select in practice.
For the local Whittle estimator
the selection of the  
bandwidth has been addressed by several authors; see for example  \cite{henry:2001}, \cite{delgado:1996} and \cite{henry:robinson:1996}. A different approach to estimate the Hurst parameter is to apply variational methods and techniques from stochastic analysis as for example derived in \cite{coeurjolly:2001} and \cite{istas:1997}. 
For an ordinal-pattern based estimation of the Hurst parameter, the asymptotic distribution of the estimator is derived on the basis of limit theorems for short-range dependent time serie in \cite{sinn:keller:2011}. 
Complementing the results of \cite{sinn:keller:2011}, we derive the limit distribution for the estimator if $H>\frac{3}{4}$. 

In \cite{fischer:schumann:schnurr:2017} the authors used ordinal patterns in the context of hydrological data. It is a well known fact that hydrological data is often long-range dependent. In this case, the limit theorems presented in \cite{schnurr:dehling:2017} are no longer valid. In the present paper we close this gap and provide limit theorems in the long-range dependent setting.

\par
For $h\in\mathbb{N}$ let $S_{h}$ denote the set of permutations of $\{0, \ldots, h\}$, which we write as $(h+1)$-tuples containing each of the numbers $0, \ldots, h$ exactly one time.
By the ordinal pattern of order $h$  we refer to the permutation
\begin{align*}
\Pi(x_0, \ldots, x_h)=(\pi_0,\ldots, \pi_h)\in S_h
\end{align*}
 which satisfies
\begin{align*}
x_{\pi_0}\geq \ldots\geq x_{\pi_h}.
\end{align*}
\noindent
 Given a time series $(\xi_j)_{j\geq 0}$, we consider the relative frequency
\[
  \hat{q}_n(\pi):=\frac{1}{n}\sum\limits_{i=0}^{n-1}1_{\left\{\Pi(\xi_i, \xi_{i+1}, \ldots, \xi_{i+h})=\pi\right\}}
\]
of an ordinal pattern $\pi\in S_h$
as  a natural estimator for the probability 
\begin{align*}
p(\pi):=\Px(\Pi(\xi_0,\ldots,\xi_{h})=\pi).
\end{align*}

\cite{sinn:keller:2011} show that Rao-Blackwellization leads to an estimator   $\hat{p}_n(\pi)$ with lower risk and therefore better statistical properties.

In this article, both estimators are studied. 
Confirming the results of 
 \cite{sinn:keller:2011}, we show   that  $\hat{q}_n(\pi)$ and $\hat{p}_n(\pi)$ are consistent estimators; see Proposition \ref{prop:consistency}.  
 We consider separately the case of a stationary time series and the case of a time series with stationary increments. While the asymptotic distribution of $\hat{q}_n(\pi)$ can be derived from a limit theorem for functions with  Hermite rank $1$, the limit behaviour of $\hat{p}_n(\pi)$ is derived from  corresponding results for functions with  Hermite rank $2$. Along the way we explicitly calculate the asymptotic distribution of partial sums of the form $\sum_{i=1}^n f(X_i,...,X_{i+p-1})$ where $f$ has Hermite rank $1$ or Hermite rank $2$ and $(X_i)_{i\geq 1}$ is a stationary long-range dependent Gaussian process.

The paper is organized as follows: in the next section we introduce the mathematical framework. In Section 3 we present the main results, namely the asymptotic properties of two estimators of ordinal pattern probabilities. In Section 4, on the basis of these considerations, the asymptotic distribution of an estimator for the Hurst parameter based on ordinal patterns is derived. The detailed proofs of more general limit theorems for functions with Hermite rank $1$ and $2$, that constitute the theoretical background of the results in Section 3 and 4, are given in Section 5. In the final section a simulation study is presented. 
\section{Mathematical Framework}\noindent
Let $(X_j)_{j\geq0}$ be a stationary standard Gaussian process with  autocovariance function
\begin{equation*}
r(k):=\Cov(X_0,X_{k})=L(k)k^{-D},~~k\geq 1,
\end{equation*}
where $L$ is a function, slowly varying at infinity (see \cite{bingham:1987}, p.6), and $0<D<1$. Such a process is called long-range dependent. For $p\in\NN$ we consider the $\RR^{p}$-valued process $(\X_j)_{j\geq 0}$ given by
\begin{equation*}
\X_j:=(\X_j^{(1)}, \X_j^{(2)}, \ldots, \X_j^{(p)}) \text{ with } \X_j^{(i)}:=X_{j+i-1},
\end{equation*}
that is, we consider overlapping finite sequences of the original process. 
For $1\leq l,m\leq p$, $p\in \mathbb{N}$, the corresponding cross-covariance function satisfies
\begin{equation*}
r^{(l,m)}(k)=\Ex\X_0^{(l)}\X_{k}^{(m)}=L(|k+m-l|)|k+m-l|^{-D}, ~~k\geq 1,
\end{equation*}
and, since $L$ is a slowly varying function, we thus obtain
\begin{equation*}
\lim_{k\rightarrow\infty}\frac{k^Dr^{(l,m)}(k)}{L(k)}=1
\end{equation*}
for all $l,m\in\mathbb{N}$. Consequently, $(\X_j)_{j\geq0}$ is multivariate long-range dependent in the sense of \cite{arcones:1994}, Section 3, if $0<D<1$. If  $D>1$, we speak of short-range dependence.
 
We recall the concept of Hermite expansion. Let $H_k$ denote the Hermite polynomial of order $k$ given by
\begin{equation*}
H_k(x)=(-1)^ke^{x^2/2}\frac{d^k}{dx^k}e^{-x^2/2},\hspace{8mm}x\in\RR,
\end{equation*}
and define the multivariate Hermite polynomial $H_{l_1,\ldots,l_p}$ by 
\begin{equation}\label{eq:def_mvH}
H_{l_1,\ldots,l_p}(x)=\prod_{i=1}^pH_{l_i}(x_i),\hspace{8mm}x\in\RR^p.
\end{equation}
The collection $(H_{l_1,\ldots,l_p})_{l_1,\ldots,l_p\geq0}$ forms an orthogonal basis of $L^2(\N(0,E_p))$, where $\N(0,E_p)$  denotes the $p$-dimensional standard normal distribution; see Section 3.2 in \cite{beran:feng:ghosh:kulik:2013}. Thus, for any square-integrable  $G:\RR^p\rightarrow\RR$ the following $L^2$-identity holds:
\begin{equation}\label{eq:Hermite_expansion}
G(\U)-\Ex G(\U)=\sum_{k=m(G,E_p)}^{\infty}\sum_{l_1+\ldots+l_p=k}\frac{J_{l_1,\ldots, l_p}}{l_1!\cdots l_p!}H_{l_1,\ldots,l_p}(\U),
\end{equation}
where $\U\sim\N(0,E_p)$. The Hermite coefficients are given by the inner product, that is $J_{l_1,\ldots, l_p}=\Ex\left(G(\U)H_{l_1,\ldots, l_p}(\U)\right)$. The starting index 
\begin{equation*}
m(G,E_p):=\min\left\{\sum_{i=1}^pl_i: J_{l_1,\ldots, l_p}\neq0\right\}
\end{equation*}
is called the Hermite rank of $G$. Since the left-hand side in \eqref{eq:Hermite_expansion} is centered, we have $m\geq1$.
In contrast to \eqref{eq:def_mvH} the definition of multivariate Hermite polynomials with respect to $\N(0,\Sigma)$ is more complicated; see \cite{beran:feng:ghosh:kulik:2013}, section 3.2. 
The Hermite rank is defined analogously
\begin{equation*}
m(G,\Sigma):=\min\left\{\sum_{i=1}^pl_i: \Ex \left(G(\X) H_{l_1,\ldots, l_p}(\X)\right)\neq0\right\},
\end{equation*}
where $\X\sim\N(0,\Sigma)$.

The Hermite expansion in \eqref{eq:Hermite_expansion} is crucial to determining the asymptotics of partial sums of the type
\begin{equation}
  \sum_{i=1}^n \left \{ f(X_i,\ldots,X_{i+p-1}) -\Ex f(X_1,\ldots,X_{p}) \right\},\label{eq:p-sum}
\end{equation}
where $f:\R^{p}\rightarrow \R$ satisfies $\Ex(f(X_1,\ldots,X_{p}))^2<\infty$.  

\section{Ordinal Patterns}
In this section we introduce the concept of ordinal pattern analysis and present asymptotic distributions for estimators of ordinal pattern probabilities, where functions with different Hermite ranks show up. We also provide examples for the calculation of the coefficients specifying the limiting distributions of these estimators for certain ordinal patterns. For detailed proofs of the given theorems the reader is referred to Section 5. 
\begin{definition}
Let $S_{h}$ denote the set of permutations of $\{0, \ldots, h\}$, which we write as $(h+1)$-tuples containing each of the numbers $0, \ldots, h$ exactly one time.
By the \emph{ordinal pattern} of order $h$  we refer to the permutation
\begin{align*}
\Pi(x_0, \ldots, x_h)=(\pi_0,\ldots, \pi_h)\in S_h
\end{align*}
 which satisfies
\begin{align*}
x_{\pi_0}\geq \ldots\geq x_{\pi_h}
\end{align*}
and $\pi_{i-1}>\pi_i$ if $x_{\pi_{i-1}}=x_{\pi_i}$ for $i=1,...,h-1$. 
\end{definition}
The latter is introduced in order to deal with ties which do not occur in our simulation study, but which might occur when dealing with real data.
\begin{remark}
Naturally, ordinal patterns are closely linked to the ranks of observations. 
Given observations $\xi_0, \ldots, \xi_h$, we define the rank $R_i$ of $\xi_i$ by
\begin{align*}
R_i:=\sum\limits_{j=0}^h 1_{\left\{\xi_i\leq \xi_j\right\}}.
\end{align*}
Note that if $\xi_i\neq \xi_j$ for all $i,j=0,\ldots,h$, $i\neq j$ then
\begin{align*}
R_i=j+1 \Leftrightarrow \pi_j=i.
\end{align*}
Thus, ranks provide a complete description of the order structure of the vector $(\xi_1, \ldots, \xi_n)$ equivalent to the description by ordinal patterns. 
\end{remark}
In this paper we are interested in estimating the probability  $p(\pi):=\Px(\Pi(\xi_0,\ldots,\xi_{h})=\pi)$ for a given time series $\XI=(\xi_t)_{t \geq 0}$ and therefore define the estimator
\[
  \hat{q}_n(\pi):= \frac{1}{n} \# \{0\leq i\leq n-1 : \Pi(\xi_i,\ldots,\xi_{i+h})=\pi  \}=\frac{1}{n}\sum\limits_{i=0}^{n-1}1_{\left\{\Pi(\xi_i, \xi_{i+1}, \ldots, \xi_{i+h})=\pi\right\}}
\]
for a time series $\left(\xi_t\right)_{t\geq 0}$.

We will see later (Remark \ref{rem:specialcase}) that the assumption that the time series $\left(\xi_t\right)_{t\geq 0}$ is stationary yields trivial limits.
Therefore, we relax this assumption and use a helpful relation that was derived in \cite{sinn:keller:2011}. They have shown that the estimator above is uniquely determined by the increments of this process. Let us consider $X_t:=\xi_t-\xi_{t-1}$ for $t \geq 1$.

For a vector $x=(x_0 \ldots, x_h)\in \mathbb{R}^{h+1}$ define
\begin{align}
\tilde{\Pi}(x_0,...,x_h):=\Pi(0, x_0, x_0+x_1, \ldots, x_0+\ldots +x_h).\label{incrementmapping}
\end{align}
Then, it holds that
\begin{align*}
\tilde{\Pi}(x_1-x_0,\ldots,x_h-x_{h-1})=\Pi(0,x_1-x_0,\ldots,x_h-x_0)
=\Pi(x_0,\ldots,x_h),
\end{align*}
since ordinal patterns are not affected by monotone transformations.\newline
In terms of random vectors we hence arrive at
\begin{align*}
\Pi(\xi_t, \xi_{t+1}, \ldots, \xi_{t+h})=\tilde{\Pi}(X_{t+1}, \ldots, X_{t+h}),~t\geq 0.
\end{align*}
\noindent
 In the following, we will study under which assumptions on the underlying time series we can derive an asymptotic result for the estimator $\hat{q}_n(\pi)$. Since regarding $\left(\xi_t\right)_{t\geq 0}$ as a stationary time series is not interesting due to the degenerate limit, we relax this assumption as follows: 
let $\XI=(\xi_t)_{t \geq 0}$ be a (possibly non-stationary) stochastic process and  let $\X=(X_t)_{t\geq 1}$ denote the corresponding increment process given by $X_t:=\xi_t-\xi_{t-1}$ for $t \geq 1$.
We assume that $\X$ is a stationary standard Gaussian process with autocovariance function
\begin{equation*}
r(k)=L(k)k^{-D},~~k\geq 1,
\end{equation*}
where $L$ is a function, slowly varying at infinity, and $0<D<1$.

We now rewrite the estimator $\hat{q}_n(\pi)$ in terms of the increment variables following the considerations in (\ref{incrementmapping}): 
\begin{align*}
\hat{q}_n(\pi)=\frac{1}{n}\sum_{i=0}^{n-1} 1_{\left\{\Pi(\xi_i,\ldots,\xi_{i+h})=\pi\right\}}=\frac{1}{n}\sum_{i=0}^{n-1} 1_{\left\{\tilde{\Pi}(X_{i+1},\ldots,X_{i+h})=\pi\right\}}
\end{align*}

We will show that the relative frequency of any ordinal pattern is a consistent estimator for the 
corresponding probability.

\begin{theorem}
Suppose that $(X_i)_{i\geq 1}$ is a stationary ergodic process. Then, $\hat{q}_n(\pi)$ is a consistent estimator of $p(\pi):=\Px(\tilde{\Pi}(X_1,\ldots,X_{h})=\pi)$. More precisely,
\[
 \lim_{n\rightarrow \infty }\hat{q}_n(\pi)=p(\pi)
\]
almost surely.
\label{th:consistent}
\end{theorem}\noindent

\subsection{Limit distribution of $\hat{q}_n(\pi)$}
At first we need to determine the Hermite rank of the estimator. Here, and in what follows, proofs are postponed to Section 5.
\begin{lemma}\label{lemm: estimatorHR1}
Let $(X_k)_{k\geq 1}$ be a stationary standard normal Gaussian process and let $h\in\NN$. Then, for any $\pi\in S_h$, the Hermite rank of
\begin{equation*}
1_{\{\tilde{\Pi}(X_1,\ldots,X_{h})=\pi\}}-\Px(\tilde{\Pi}(X_1,\ldots,X_{h})=\pi)
\end{equation*}
is equal to 1.
\end{lemma}

We now give the asymptotic distribution of the estimator and in doing so, we will take a closer look at the Hermite coefficients which determine the limit variance and hence the limit distribution. 

\begin{theorem}\label{th:lt-he}
Let $\XI=(\xi_t)_{t\geq 0}$ be a stochastic process and  let $\X=(X_t)_{t\geq 1}$ denote the increment process of $\XI$ given by $X_t:=\xi_t-\xi_{t-1}$ for $t\in \mathbb{Z}$.
Assume that $\X$ is a stationary, long-range dependent standard Gaussian process with  autocovariance function
$r(k)=L(k)k^{-D}$. Then,
\[
  n^{D/2}L^{-1/2}(n)\left(\hat{q}_n(\pi)
  -\Px\left(\tilde{\Pi}(X_{1},\ldots,X_{h})=\pi\right)\right) \xrightarrow{D} \N \left(0,c_D \left(\sum_{j=1}^{h} \alpha_j\right)^2\right),
\]
where $c_D=\frac{2}{(1-D)(2-D)}$ and where the vector $\alpha=(\alpha_1,\ldots,\alpha_{h})^t$ is given by
\[
  \alpha:=\Sigma_p^{-1} c
\]
with $c=(c_1,\ldots,c_h)^t$ defined by
\[
  c_k=\Ex \left\{1_{\left\{\tilde{\Pi}(X_1, \ldots, X_{h})=\pi\right\}} X_k\right\}, \; 1\leq k\leq h.
\]
\end{theorem}
\noindent
Thus, in order to compute the limit variance of $\hat{q}_n(\pi)$, we have to calculate the constants $c_k$ for $k=1,...,h$. 
We can reduce the number of calculations by making use of
the time and space symmetry of stationary multivariate normal random vectors. For a normal random vector $\left(X_1,\ldots,X_{h}\right)$ these are given by
\begin{align*}
\left(X_1,\ldots,X_{h}\right)&\overset{D}{=}\left(-X_1,\ldots,-X_{h}\right),\\
\left(X_1,\ldots,X_{h}\right)&\overset{D}{=}\left(X_{h},...,X_{1}\right).
\end{align*}
\par
Following \cite{sinn:keller:2011}, p. 1784, we define two mappings:\par
\begin{align*}
&\mathcal{S}:S_{h}\rightarrow S_{h}, \left(\pi_0,\ldots,\pi_{h}\right)\mapsto \left(\pi_h,\ldots,\pi_0\right),\\
&\mathcal{T}:S_{h}\rightarrow S_{h}, \left(\pi_0,\ldots,\pi_{h}\right)\mapsto \left(h-\pi_0,\ldots,h-\pi_h\right),
\end{align*}
\noindent
\begin{figure}[h]
\begin{center}
\begin{tikzpicture}
[scale =1]
\begin{axis}[ height =3cm,
  width=10cm,
axis x line=none,
axis y line=none
]
\addplot [mark=*,blue,line width =1.5pt] coordinates {
(2,0.1)(3,0.6)(4,0.4)(5,0.5)};
\addplot [mark=*,blue,line width =1.5pt] coordinates {
(7,0.5)(8,0.4)(9,0.6)(10,0.1)};
\end{axis}
\end{tikzpicture}
\newline
 $\pi=(1,3,2,0)$  \quad \quad \quad $\mathcal{T}(\pi)=(2,0,1,3)$ \quad \quad \quad  \quad \quad \quad\quad
\end{center}
\end{figure}
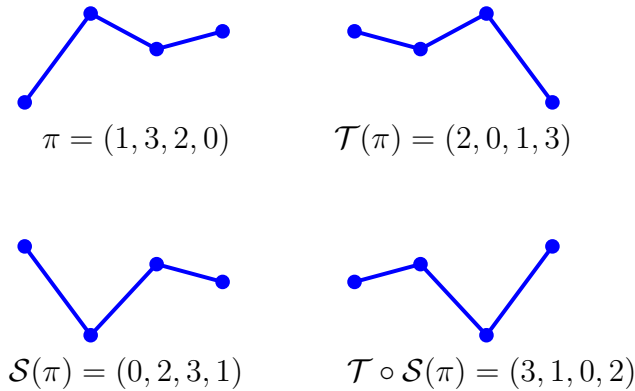
\begin{figure}[h]
\begin{center}
\begin{tikzpicture}
[scale =1]
\begin{axis}[ height =3cm,
  width=10cm,
axis x line=none,
axis y line=none
]
\addplot [mark=*,blue,line width =1.5pt] coordinates {
(2,0.6)(3,0.1)(4,0.5)(5,0.4)};
\addplot [mark=*,blue,line width =1.5pt] coordinates {
(7,0.4)(8,0.5)(9,0.1)(10,0.6)};
\end{axis}
\end{tikzpicture}
\newline
$\mathcal{S}(\pi)=(0,2,3,1)$ \quad \quad \quad   $\mathcal{T}\circ\mathcal{S}(\pi)=(3,1,0,2)$  \quad \quad \quad  \quad \quad 
\caption{Space and time reversion of the pattern $\pi=(1,3,2,0)$.}
\label{fig: figure 1}
\end{center}
\end{figure}~

\noindent
Graphically, the mapping $\mathcal{S}$ can be considered as  space reversal, i.e., as
 the reflection of $\pi$ on a horizontal line, while $ \mathcal{T}$ can be considered as time reversal, i.e., 
as the reflection of $\pi$ on a vertical line.\\
For each $\pi\in S_{h}$, we define 
\begin{align}
\bar{\pi}:=\{\pi,\mathcal{S}(\pi),\mathcal{T}(\pi),\mathcal{T}\circ\mathcal{S}(\pi)\}.
\end{align}
It is easily seen that the set $\bar{\pi}$ is closed under $\mathcal{S}$ and $\mathcal{T}$, since $\mathcal{S}\circ\mathcal{S}(\pi)=\mathcal{T}\circ\mathcal{T}(\pi)=\pi$ and $\mathcal{T}\circ\mathcal{S}(\pi)=\mathcal{S}\circ\mathcal{T}(\pi)$. This yields a partition of $S_h$ into sets each having either two or four elements, depending on whether $\mathcal{T}(\pi)=\mathcal{S}(\pi)$ holds for the considered $\pi$.

In \cite{sinn:keller:2011}, p.1786 and Lemma 1, it is shown that with respect to ordinal patterns the above considerations yield
\begin{align}
&\Ex\left(X_k1_{\left\{\tilde{\Pi}\left(X_1,...,X_{h}\right)=\pi\right\}}\right)=-\Ex\left(X_k1_{\left\{\tilde{\Pi}\left(X_1,...,X_{h}\right)=\mathcal{S}(\pi)\right\}}\right)\label{eq:spacereversion},~~k=1,...,h, \\
&\Ex\left(X_k1_{\left\{\tilde{\Pi}\left(X_1,...,X_{h}\right)=\pi\right\}}\right)=-\Ex\left(X_{h+1-k} 1_{\left\{\tilde{\Pi}\left(X_1,...,X_{h}\right)=\mathcal{T}(\pi)\right\}}\right)\label{eq:timereversion},~~k=1,...,h.
\end{align}
Both equations follow from the space and time symmetry of the multivariate normal distribution. More precisely, \eqref{eq:timereversion} holds since ordinal patterns are not affected by monotone transformations. For $\pi\in S_h$ we have
\begin{align*}
\left\{\tilde{\Pi}\left(X_1,...,X_{h}\right)=\mathcal{T}(\pi)\right\}
&=\left\{\mathcal{T}\left(\tilde{\Pi}(X_1,\ldots,X_h)\right)=\pi\right\}\\
&=\left\{\mathcal{T}\left(\Pi(0,X_1,X_1+X_2,\ldots,X_1+\ldots+X_h)\right)=\pi\right\}\\
&=\left\{\Pi(X_1+\ldots+X_h,\ldots,X_1+X_2,X_1,0)=\pi\right\}\\
&=\left\{\Pi(0,-X_h,-(X_h+X_{h-1}),\ldots,-(X_1+...+X_h))=\pi\right\}\\
&=\left\{\tilde{\Pi}(-X_{h},-X_{h-1},\ldots,-X_2,-X_1)=\pi\right\}.
\end{align*}
We compute the limit variance for ordinal patterns of lengths $p=2$ and $p=3$, i.e., we need to study increments of length $h=1$ and $h=2$.
As it is common in the literature, we restrict ourselves to small $h$ in the present article. Unfortunately, the computations for larger values of $h$ exceed the computing capacity of Mathematica.\par
Given the symmetry relations in \eqref{eq:spacereversion} and \eqref{eq:timereversion}, we only need to calculate the Hermite coefficients of the estimator $\hat{q}_n(\pi)$ for one pattern $\pi$ of each reversion group. Regarding $S_1=\{(0,1),(1,0)\}$ it is sufficient to choose $(1,0)$. Regarding $S_2$ we can partition this set into the two subsets $\{(2,1,0),(0,1,2)\}$ and $\{(2,0,1),(0,2,1),(1,2,0),(1,0,2)\}$. In the following we will study the Hermite coefficients of $\hat{q}_n(\pi)$ for $\pi=(2,1,0)$ and $\pi=(2,0,1)$ so that we can reduce the number of lengthy calculations since we only need to consider two ordinal patterns instead of six.
\begin{example}[Ordinal patterns of length $p=2$]
In the case $h=1$  there are only two possible patterns: $\pi=(0,1)$ and the corresponding spatial (or time) reverse $\pi=(1,0)$. We focus on $\pi=(1,0)$.  This pattern corresponds to the event $\left\{\Pi(\xi_0,\xi_1)=(1,0)\right\}=\left\{\xi_1\geq \xi_0\right\}=\left\{X_1\geq 0\right\}$. Hence, we consider
\begin{align*}
c_1=\Ex\left(X_1 1_{\left\{X_1\geq 0\right\}}\right)=\int_0^{\infty}y_1 \varphi(y_1)\mathrm{d}y_1=\varphi(0).
\end{align*}
Correspondingly, we obtain $c_1=-\varphi(0)$ for $\pi=(0,1)$ since this is the spatial reversion of $(1,0)$.
Thus, for these two ordinal patterns we arrive at a limit distribution of $q_n(\pi)$ given by $\mathcal{N}\left(0,c_D\varphi^2(0)\right)$, where $c_D=\frac{2}{(1-D)(2-D)}$.
\end{example}

We continue with the calculation of the limit variances in the case $p=3$. The integrals under consideration were solved by using Mathematica as well as a lengthy calculations that make use of  the Cholesky decomposition (cf. the Appendix).

\begin{example}[Ordinal patterns of length $p=3$]\label{ex: forcholesky}
First, we study the limit variance for $\pi=(2,1,0)$. In this case, $\bar{\pi}$ has two elements. Note that $\{\Pi(\xi_0,\xi_1,\xi_2)=(2,1,0)\}=\left\{\xi_2\geq \xi_1\geq \xi_0\right\}=\left\{X_2\geq 0,X_1\geq 0\right\}$. Due to the symmetry of the bivariate normal distribution, we obtain $c_1=c_2$,  so that we only need to calculate 
\begin{align*}
c_1=\Ex\left(X_1 1_{\left\{X_2\geq 0,X_1\geq 0\right\}}\right)=\int_0^{\infty}\int_0^{\infty} y_1 \varphi_{\left(X_1,X_2\right)}(y_1,y_2)\mathrm{d}y_1\mathrm{d}y_2=\frac{\varphi(0)}{2}\left(1+r(1)\right),
\end{align*}
where $\varphi_{\left(X_1,X_2\right)}$ denotes the joint density of $(X_1,X_2)$ .
Hence,
\begin{align*}
\sum_{j=1}^2 \alpha_j=2c_1\left(g_{1,1}+g_{2,1}\right)=2c_1\frac{1-r(1)}{1-\left(r(1)\right)^2}=\varphi(0),
\end{align*}
where $g_{i,j}$ are the entries of $\Sigma_2^{-1}$ given by
\begin{align*} \Sigma_2^{-1}=\frac{1}{1-(r(1))^2}\left( 
 \begin{array}{rr}
 1 & -r(1)  \\
  -r(1) & 1
 \end{array}  
   \right).
 \end{align*}
Again, we obtain the limit variance $c_D\varphi^2(0)$ which is here more surprising than in the case $h=1$ because the result is independent of $r(1)$.
For the space reverse pattern $\pi_2=(0,1,2)$ we apply \eqref{eq:spacereversion} and obtain $c_1=-\phi(0)$ leading to the same limit variance.
It is an interesting question whether it is just a coincidence that this variance is independent of the covariance between the increments. The answer turns out to be yes, since the dependence is reflected in the limit variance of the pattern $\pi=(2,0,1)$.

Note that $\left\{\Pi(\xi_0,\xi_1,\xi_2)=(2,0,1)\right\}=\left\{\xi_1\leq \xi_0\leq \xi_2\right\}=\left\{X_1\leq 0,X_1+X_2\geq 0\right\}=\left\{X_1\leq 0,X_2\geq -X_1\right\}$. As a result, we have
\begin{align*}
c_1&=\Ex\left(X_1 1_{\left\{X_1\leq 0,X_2\geq -X_1\right\}}\right)=\int_{-\infty}^0 \int_{-y_1}^{\infty} y_1 \varphi_{\left(X_1,X_2\right)}(y_1,y_2)\mathrm{d}y_2\mathrm{d}y_1\\
&=\frac{\varphi(0)}{2}\left(\frac{\sqrt{1+r(1)}}{\sqrt{2}}-1\right),\\
c_2&=\Ex\left(X_2 1_{\left\{X_1\leq 0,X_2\geq -X_1\right\}}\right)=\int_{-\infty}^0 \int_{-y_1}^{\infty} y_2 \varphi_{\left(X_1,X_2\right)}(y_1,y_2)\mathrm{d}y_2\mathrm{d}y_1\\
&=\frac{\varphi(0)}{2}\left(\frac{\sqrt{1+r(1)}}{\sqrt{2}}-r(1)\right),
\end{align*}
where $\varphi_{\left(X_1,X_2\right)}$ denotes the joint density of $(X_1,X_2)$.
As a result,  we obtain 
\begin{align*}
\sum_{j=1}^2 \alpha_j=\left(c_1+c_2\right)\left(g_{1,1}+g_{2,1}\right)&=\frac{\varphi(0)}{2}\left(\frac{\sqrt{2(1+r(1))}-(1+r(1))}{1+r(1)}\right)\\
&=\frac{\varphi(0)}{2}\left(\frac{\sqrt{2}}{\sqrt{1+r(1)}}-1\right).
\end{align*}
The above expression depends on $r(1)$. Due to space and time symmetry discussed in \eqref{eq:spacereversion} and \eqref{eq:timereversion} all permutations that belong to the reversion group of $\pi=(2,0,1)$, i.e., $(1,0,2),(0,2,1)$ and $(1,2,0)$, lead to the same limit distribution for $\hat{q}_n(\pi)$, namely 
\[
\mathcal{N}\left(0,c_D\left(\frac{\varphi(0)}{2}\left(\frac{\sqrt{2}}{\sqrt{1+r(1)}}-1\right)\right)^2\right).
\]
\end{example}
\subsection{Limit distribution of an improved estimator based on Rao-Blackwellization}
In the previous section we considered the natural estimator for the frequency of a certain ordinal pattern. However,  in \cite{sinn:keller:2011} it is shown that the estimator which results from averaging  the  estimates of the same reversion class  has better statistical properties. The corresponding estimator is therefore defined by
\begin{align*}
\hat{p}_n(\pi):= \frac{1}{n} \sum\limits_{t=0}^{n-1} \frac{1}{\# \bar{\pi}}1_{\left\{\Pi(\xi_t, \xi_{t+1}, \ldots, \xi_{t+h})\in \bar{\pi}\right\}}, 
\end{align*}
where $\#\bar{\pi}$ denotes the cardinality of the set $\bar{\pi}$.

\noindent
Recalling that $\Pi(\xi_t, \xi_{t+1}, \ldots, \xi_{t+h})=\tilde{\Pi}(X_{t+1},\ldots, X_{t+h})$, we are, in particular, interested 
in the function $f:\mathbb{R}^h\longrightarrow \mathbb{R}$ defined by
\begin{align}\label{eq: function improved estimator}
f(x_1,\ldots,x_h):=\frac{1}{\# \bar{\pi}}1_{\{\tilde{\Pi}(x_1,\ldots,x_h)\in \bar{\pi}\}}-\frac{1}{\# \bar{\pi}}P\left(\tilde{\Pi}(x_1,\ldots,x_h)\in \bar{\pi}\right).
\end{align}
In order to specify the limit distribution of $\hat{p}_n(\pi)$, we need to determine the Hermite rank of this function. For this,  note that \cite{sinn:keller:2011}, p. 1786, show that $f$ has Hermite rank $m\geq 2$.

For a multivariate random vector $\left(X_1,...,X_{h}\right)\sim\mathcal{N}(0,\Sigma_h)$ define
\begin{align*}
c_{i,i}^{\pi}:=\Ex\left[\left(X_i^2-1\right) 1_{\left\{\tilde{\Pi}\left(X_1,...,X_{h}\right)=\pi\right\}}\right]&=\Ex\left[\left(X_i^2-1\right) 1_{\left\{\tilde{\Pi}\left(X_1,...,X_{h}\right)=\mathcal{S}(\pi)\right\}}\right]\\
&=\Ex\left[\left(X_{h+1-i}^2-1\right) 1_{\left\{\tilde{\Pi}\left(X_1,...,X_{h}\right)=\mathcal{T}(\pi)\right\}}\right],~i=1,...,h.
\end{align*}
Analogously, we obtain
\begin{align*}
c_{i,j}^{\pi}&:=\Ex\left[ \left(X_i X_j - \Ex\left(X_i X_j\right)\right) 1_{\left\{\tilde{\Pi}\left(X_1,...,X_{h}\right)=\pi\right\}}\right]\\
&=\Ex\left[ \left(X_i X_j - \Ex\left(X_i X_j\right)\right)1_{\left\{\tilde{\Pi}\left(X_1,...,X_{h}\right)=\mathcal{S}(\pi)\right\}}\right]\\
&=\Ex\left[ \left(X_{h+1-i} X_{h+1-j} - \Ex\left(X_{h+1-i} X_{h+1-j}\right)\right)1_{\left\{\tilde{\Pi}\left(X_1,...,X_{h}\right)=\mathcal{T}(\pi)\right\}}\right],
\end{align*}
$i,j=1,...,h,~i\neq j,$ 
so that alltogether we derive 
\begin{align}\label{eq: relations hermite coefficients}
c_{i,j}^{\pi}=c_{i,j}^{\mathcal{S}(\pi)}=c_{h+1-i,h+1-j}^{\mathcal{T}(\pi)}=c_{h+1-i,h+1-j}^{\mathcal{T}\circ \mathcal{S}(\pi)},~i,j=1,...,h.
\end{align}
With this result we  can simplify the second order Hermite coefficients for the improved estimator
\begin{align*}
c_{i,j}:=\Ex\left[ \left(X_i X_j - \Ex\left(X_i X_j\right)\right) f\left(X_1,...,X_{h}\right)\right]&=\frac{1}{\#\bar{\pi}}\sum\limits_{\pi\in\bar{\pi}}\Ex\left[	 \left(X_i X_j - \Ex\left(X_i X_j\right)\right)   1_{\left\{\tilde{\Pi}\left(X_1,...,X_{h}\right)=\pi\right\}}     \right]\\
&=\frac{1}{\#\bar{\pi}}\sum\limits_{\pi\in\bar{\pi}}c_{i,j}^{\pi}\\
&=\frac{1}{2}\left(c_{i,j}^{\pi}+c_{h+1-i,h+1-j}^{\pi}\right).
\end{align*}
Analogously, we obtain
\begin{align*}
c_{i,i}=\frac{1}{2}\left(c_{i,i}^{\pi}+c_{h+1-i,h+1-i}^{\pi}\right).
\end{align*}
Hence, we can uniquely determine the second order Hermite coefficients of the improved estimator  by calculating the second order Hermite coefficients for only one pattern $\pi$ that belongs to the considered reversion group $\bar{\pi}$. By following the symmetry properties discussed above we derive for the special case $\mathcal{T}\circ\mathcal{S}(\pi)=\pi$ 
\begin{align*}
c_{i,j}=c_{i,j}^{\pi},~~\pi\in\bar{\pi}
\end{align*}
for all $i,j=1,...,h$.

The second order Hermite coefficients of the improved estimator $\hat{q}_n(\pi)$ are equal to the second order Hermite coefficients of $\hat{p}_n(\pi)$.\newline
We use this result to determine the Hermite rank of the function $f$ defined in \eqref{eq: function improved estimator}, for details see Section 5, and to simplify the calculations concerning the parameters determining the variance in the next Theorem \ref{th:lt-ie}.

\begin{lemma}\label{lemm: estimatorHR2}
The function
$f(x_1,...,x_h):=\frac{1}{\# \bar{\pi}}1_{\{\tilde{\Pi}(x_1,...,x_h)\in \bar{\pi}\}}-\frac{1}{\# \bar{\pi}}\Px\left(\tilde{\Pi}(x_1,...,x_h)\in \bar{\pi}\right)$ has Hermite rank $m(f,\Sigma_h)=2$.
\end{lemma}

\begin{remark}
By a similar calculation we obtain that $b^{\pi}_{jj}=0$ for all $j=1,...,h$ for the fixed pattern in the setting above.
\end{remark}

Following the above Lemma,
we derive the asymptotic distribution of the new estimator:
\begin{theorem}\label{th:lt-ie}
Let $\XI=(\xi_t)_{t\geq 0}$ be a stochastic process and  let $\X=(X_t)_{t\geq 1}$ denote the process of increments of $\XI$ given by $X_t:=\xi_t-\xi_{t-1}$ for $t\geq 1$.
Assume that $\X$ is a stationary, long-range dependent standard Gaussian process with  autocovariance function
$r(k)=L(k)k^{-D}$.  Then, if $D\in (0, \frac{1}{2})$,
\begin{align}\label{convergence improved estimator}
n^{D}(2!C_2)^{-\frac{1}{2}}L^{-1}(n) \left(\hat{p}_n(\pi)-\Px(\tilde{\Pi}(X_1,...,X_h)=\pi)\right) \overset{\mathcal{D}}{\longrightarrow}Z_{2, (1-D/2)}(1)
\sum\limits_{j=1}^{h}\sum\limits_{k=1}^{h}\alpha_{j, k},
\end{align}
with 
$C_2=\left((1-2D)(2-D)\right)^{-1}$, 
$\left(\alpha_{l,k}\right)_{1\leq l,k\leq h}=\Sigma_{h}^{-1}C\Sigma_{h}^{-1}$ and  
\begin{align*}
C=\Ex\left((X_1,...,X_h)\frac{1}{\#\bar{\pi}}\left[1_{\left\{\tilde{\Pi}(X_1,...,X_h)\in\bar{\pi}\right\}}-\Px(\tilde{\Pi}(X_1,...,X_h)\in\bar{\pi})\right](X_1,...,X_h)^t\right).
\end{align*}
\end{theorem}

\begin{remark}
For $D>\frac{1}{2}$, the asymptotic distribution of $\hat{p}_n(\pi)$ is derived in \cite{keller:sinn:2005}, Theorem 7. In this case, it is Gaussian.
\end{remark}

\noindent
For small $h$ we calculate the matrix of coefficients $\left(\alpha_{l,k}\right)_{1\leq l,k\leq h}$ explicitly:

\begin{example}[The case $h=1$]
Since we are interested in increments with length $h=1$, we have to study ordinal patterns of length $p=2$. Regarding $\pi=(1,0)$ we derive the event $\left\{\Pi(\xi_0,\xi_1)=(1,0)\right\}=\left\{\xi_0\leq\xi_1\right\}=\left\{X_1\geq 0\right\}$ and therefore
\begin{align*}
c_{1,1}=\Ex\left[\left(X_1^2-1\right)1_{\left\{X_1\geq 0\right\}}\right]=\int_0^{\infty}\left(y_1^2-1\right)\varphi(y_1)\mathrm{d}y_1=0.
\end{align*}
So in the \textit{trivial} case (only one increment variable) we derive a degenerate limit distribution again.
\end{example}
For increments of length $h=2$, we used  Mathematica to calculate the Hermite coefficients.
\begin{example}[The case $h=2$]\label{ex: coeff}
First, we consider the pattern $\pi=(2,1,0)$ and the corresponding event $\left\{\Pi(\xi_0,\xi_1,\xi_2)=(2,1,0)\right\}=\left\{\xi_2\geq \xi_1 \geq \xi_0\right\}=\left\{X_1\geq 0, X_2\geq 0\right\}$. We know that $c_{i,j}=c_{i,j}^{\pi}$, $i,j=1,2$, and by \eqref{eq: relations hermite coefficients} that $c_{1,1}=c_{2,2}$ since $\mathcal{T}\circ\mathcal{S}(2,1,0)=(2,1,0)$.
We have
\begin{align*}
c_{1,1}=\Ex\left[\left(X_1^2-1\right)1_{\left\{X_1\geq 0,X_2\geq 0\right\}}\right]&=\int_0^{\infty}\int_0^{\infty}\left(y_1^2-1\right)\varphi_{\left(X_1,X_2\right)}(y_1,y_2)\mathrm{d}y_1\mathrm{d}y_2\\
&=\varphi^2(0)r(1)\sqrt{1-\left(r(1)\right)^2}
\end{align*}
and
\begin{align*}
c_{1,2}&=\Ex\left[\left(X_1X_2-\Ex\left(X_1 X_2\right)\right)1_{\left\{X_1\geq 0,X_2\geq 0\right\}}\right]\\
&=\int_0^{\infty}\int_0^{\infty}y_1 y_2\varphi_{\left(X_1,X_2\right)}(y_1,y_2)\mathrm{d}y_1\mathrm{d}y_2-r(1)\int_0^{\infty}\int_0^{\infty}\varphi_{\left(X_1,X_2\right)}(y_1,y_2)\mathrm{d}y_1\mathrm{d}y_2\\
&=\varphi^2(0)\sqrt{1-\left(r(1)\right)^2}.
\end{align*}
This yields
\begin{align*}
\sum_{i,j=1}^2 \alpha_{i,j}&=2\left(g_{1,2}+g_{2,2}\right)^2\left(c_{1,1}+c_{1,2}\right)\\
&=2\frac{c_{1,1}+c_{1,2}}{\left(1+r(1)\right)^2}\\
&=2\varphi^2(0)\sqrt{\frac{1-r(1)}{1+r(1)}}.
\end{align*}
For $\pi=(2,1,0)$ the left-hand side in \eqref{convergence improved estimator} converges in distribution to $2\varphi^2(0)\sqrt{\frac{1-r(1)}{1+r(1)}}Z_{2,H}(1)$.\par
Consider the pattern $\pi=(2,0,1)$ and the corresponding event $\left\{\Pi(\xi_0,\xi_1,\xi_2)=(2,0,1)\right\}=\left\{\xi_1 \leq \xi_0 \leq \xi_2\right\}=\left\{X_1\leq 0, X_1+X_2\geq 0\right\}$. It holds that
\begin{align*}
c_{1,1}^{\pi}=\Ex\left[\left(X_1^2-1\right)1_{\left\{X_1\leq 0,X_2\geq -X_1\right\}}\right]&=\int_{-\infty}^0 \int_{-y_1}^{\infty}\left(y_1^2-1\right)\varphi_{\left(X_1,X_2\right)}(y_1,y_2)\mathrm{d}y_2\mathrm{d}y_1\\
&=-\varphi^2(0)\frac{\sqrt{1-\left(r(1)\right)^2}}{2} 
\end{align*}
and
\begin{align*}
c_{1,2}^{\pi}&=\Ex\left[\left(X_1X_2-\Ex\left(X_1 X_2\right)\right)1_{\left\{X_1\leq 0,X_2\geq -X_1\right\}}\right]\\
&=\int_{-\infty}^0 \int_{-y_1}^{\infty}y_1 y_2\varphi_{\left(X_1,X_2\right)}(y_1,y_2)\mathrm{d}y_2\mathrm{d}y_1-r(1)\int_{-\infty}^0 \int_{-y_1}^{\infty}\varphi_{\left(X_1,X_2\right)}(y_1,y_2)\mathrm{d}y_2\mathrm{d}y_1\\
&=-\varphi^2(0)\frac{\sqrt{1-\left(r(1)\right)^2}}{2}.
\end{align*}
Since the reversion group of this pattern has four elements we also need to calculate
\begin{align*}
c_{2,2}^{\pi}=\Ex\left[\left(X_2^2-1\right)1_{\left\{X_1\leq 0,X_2\geq -X_1\right\}}\right]&=\int_{-\infty}^0 \int_{-y_1}^{\infty}\left(y_2^2-1\right)\varphi_{\left(X_1,X_2\right)}(y_1,y_2)\mathrm{d}y_2\mathrm{d}y_1\\
&=-\varphi^2(0)\frac{\sqrt{1-\left(r(1)\right)^2}(2r(1)-1)}{2}.
\end{align*}

Altogether we arrive at
\begin{align*}
\sum_{i,j=1}^2 \alpha_{i,j}
&=\frac{1}{\left(1+r(1)\right)^2}\left(c_{1,1}+2c_{1,2}+c_{2,2}\right)\\
&=\frac{1}{\left(1+r(1)\right)^2}\left(c_{1,1}^{\pi}+2c_{1,2}^{\pi}+c_{2,2}^{\pi}\right)\\
&=-\varphi^2(0)\frac{\sqrt{1-\left(r(1)\right)^2}}{\left(1+r(1)\right)^2}\left(r(1)+1\right)\\
&=-\varphi^2(0)\sqrt{\frac{1-r(1)}{1+r(1)}}.
\end{align*}
For $\pi=(2,0,1)$ the left-hand side in \eqref{convergence improved estimator} converges in distribution to $-\varphi^2(0)\sqrt{\frac{1-r(1)}{1+r(1)}}Z_{2,H}(1)$.
\end{example}\par
\begin{remark} \label{rem:specialcase}
The reader might wonder which limit theorems one can derive in the special case that it is not only the increment process which is stationary but the time series itself.
We have to determine the Hermite rank of the estimator $\hat{q}_n(\pi)$ in this setting and we obtain that for any $\pi\in S_h$ the Hermite rank of the function $f:\mathbb{R}^{h+1} \longrightarrow \mathbb{R}$, defined by
\begin{equation*}
f(x_0, x_1, \ldots, x_h):=1_{\{\Pi(x_0,\ldots,x_{h})=\pi\}}-\Px(\Pi(x_0,\ldots,x_{h})=\pi),
\end{equation*}
is equal to 1 (for details, see Section 5).\par
We get the following asymptotic result concerning the ordinal pattern probability estimator $\hat{q}_n(\pi)$ in this modified setting:
\[
  n^{D/2} L^{-1/2}(n) \left(\hat{q}_n(\pi)
  -\Px\left(\Pi(X_0,\ldots,X_h)=\pi\right)\right) \xrightarrow{D} \delta_0,
\] 
where $\delta_0$ denotes the Dirac measure in $0$. In this special case, the limit distribution for $\hat{q}_n(\pi)$ is  trivial.
\end{remark}
However, taking the classical rate of convergence $n^{1/2}$, we will get a non-trivial Gaussian central limit theorem as explained in section 2.1.

\section{Estimation of the Hurst parameter}\noindent
\cite{sinn:keller:2011} derive an estimator for the Hurst parameter based on the improved estimator for ordinal pattern probabilites $\hat{p}_n(\pi)$. They show asymptotic normality of this estimator in the case $H<\frac{3}{4}$.
In order to obtain the asymptotic distribution for $H>\frac{3}{4}$, we briefly describe the setting that was developed in that article. The idea is to determine the probability of changes in the ``up-and-down" behaviour of the process $\XI$. Since we need to use \textit{orthant probabilites} of the normal distribution, we restrict ourselves to the case $h=2$ here.
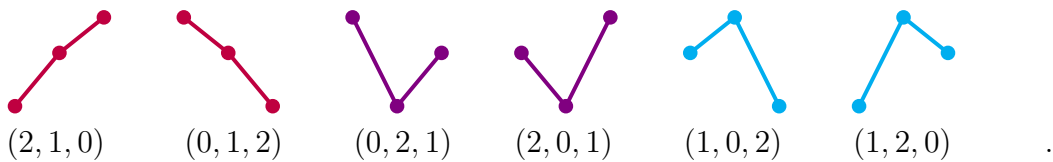
\begin{figure}[h]
\begin{center}
\begin{tikzpicture}
[scale =1]
\begin{axis}[ height =3cm,
  width=3cm,
axis x line=none,
axis y line=none
]
\addplot [mark=*,purple,line width =1.5pt] coordinates {
(0,0.1)(1,0.4)(2,0.6)};
\end{axis}
\end{tikzpicture}\quad\quad
\begin{tikzpicture}
[scale =1]
\begin{axis}[ height =3cm,
  width=3cm,
axis x line=none,
axis y line=none
]
\addplot [mark=*,purple,line width =1.5pt] coordinates {
(0,0.6)(1,0.4)(2,0.1)};
\end{axis}
\end{tikzpicture}\quad\quad
\begin{tikzpicture}
[scale =1]
\begin{axis}[ height =3cm,
  width=3cm,
axis x line=none,
axis y line=none
]
\addplot [mark=*,violet,line width =1.5pt] coordinates {
(0,0.6)(1,0.1)(2,0.4)};
\end{axis}
\end{tikzpicture}\quad\quad
\begin{tikzpicture}
[scale =1]
\begin{axis}[ height =3cm,
  width=3cm,
axis x line=none,
axis y line=none
]
\addplot [mark=*,violet,line width =1.5pt] coordinates {
(0,0.4)(1,0.1)(2,0.6)};
\end{axis}
\end{tikzpicture}\quad\quad
\begin{tikzpicture}
[scale =1]
\begin{axis}[ height =3cm,
  width=3cm,
axis x line=none,
axis y line=none
]
\addplot [mark=*,cyan,line width =1.5pt] coordinates {
(0,0.4)(1,0.6)(2,0.1)};
\end{axis}
\end{tikzpicture}\quad\quad
\begin{tikzpicture}
[scale =1]
\begin{axis}[ height =3cm,
  width=3cm,
axis x line=none,
axis y line=none
]
\addplot [mark=*,cyan,line width =1.5pt] coordinates {
(0,0.1)(1,0.6)(2,0.4)};
\end{axis}
\end{tikzpicture}\newline
$(2,1,0)$\quad\quad\thickspace\thickspace$(0,1,2)$\quad\quad\thickspace$(0,2,1)$\quad\quad$(2,0,1)$\quad\quad\thickspace$(1,0,2)$\quad\thickspace\quad$(1,2,0)$\quad\quad\quad.
\caption{Ordinal patterns for $h=2$.}
\label{fig: Bild 10}~
\end{center}
\end{figure}~\newline
To capture this mathematically, we define 
\begin{align*}
W(i):=1_{\left\{\Pi(\xi_i,...,\xi_{i+2})\in\bar{\pi}\right\}}
\end{align*}
with $\bar{\pi}=\{(2,0,1),(1,0,2),(0,2,1),(1,2,0)\}$.

Therefore, we obtain
\begin{align*}
c:=\Px(W(i)=1)=2\Px\left(X_{i+1}\geq 0,X_{i+2}\leq 0\right)=\frac{1}{2}-\frac{1}{\pi}\arcsin(r(1)),
\end{align*}
where $r$ is the covariance function of the stationary and long-range dependent increment process $\X=(X_k)_{k\geq 1}$ of $\XI$ as defined above;  see \cite{kotz:2004}, p.92. Since $r$ depends on the long-range dependence parameter $D$, which we can express as $D=2-2H$ in terms of the Hurst parameter, we will write $c=c(H)$ in the following.

In order to estimate this probability, we choose the relative frequency as an estimator:
\begin{align*}
\hat{c}_n:=\frac{1}{n}\sum_{i=0}^{n-1}W(i)=4\hat{p}_{n}(\pi),
\end{align*}
with $\pi\in\{(2,0,1),(1,0,2),(0,2,1),(1,2,0)\}$.
We want to estimate the Hurst parameter $H$ in the case that $\X$ is fractional Gaussian noise and hence $\XI$ is fractional Brownian motion. The correlation function of fractional Gaussian noise is given by
\begin{align*}
r_H(k)=\frac{1}{2}\left[(k+1)^{2H}-2k^{2H}+(k-1)^{2H}\right]
\end{align*}
such that $r_H(1)=2^{2H-1}-1$.
Therefore, we obtain
\begin{align*}
c(H)=1-\frac{2}{\pi}\arcsin(2^{H-1}),~H\in(0,1),
\end{align*}
since $\arcsin(x)=2\arcsin\left(\sqrt{\frac{1+x}{2}}\right)-\frac{\pi}{2}$ for $x\in[-1,1]$.
The probability of changes in the up-and-down-behaviour gets smaller if the Hurst parameter gets larger, as expected  intuitively due to the persistent behaviour of long-range dependent time series.
We calculate the inverse of $c$ by
\begin{align*}
g(x):=\max\left\{0,\log_2\left(\cos\left(\frac{\pi x}{2}\right)\right)+1\right\}, x\in\left[0,1\right],
\end{align*}
so that $H=g(c(H))$ is satisfied

The \textit{Zero-Crossing} estimator of the Hurst Parameter $H$ is then defined by
\begin{align*}
\hat{H}_n:=g(\hat{c}_n).
\end{align*}

In \cite{sinn:keller:2011}, Corollary 11, it is shown that $\hat{H}_n$ is a strongly consistent and asymptotically unbiased estimator of the Hurst Parameter, as well as it is asymptotically normal if $H<\frac{3}{4}$. Using Theorem \ref{th:lt-ie} we can complement their result by the following theorem.
\begin{theorem}\label{th:lt-hurst}
If $H>\frac{3}{4}$, 
\begin{align*}
n^{2-2H}\frac{\sqrt{4H-3}}{\sqrt{H}(2H-1)}\left(\hat{H}_n-H\right)\xrightarrow{D} Z_{2, H}(1)\left( \frac{2\pi}{2\log2}\tan\left(\frac{\pi c(H)}{2}\right)
\varphi^2(0)\sqrt{2^{2-2H}-1}\right).
\end{align*}
\end{theorem}
\begin{proof}
Since $\hat{c}_n=4\hat{p}_n$, it follows by Theorem \ref{th:lt-ie} and Example \ref{ex: coeff} that
\begin{align*}
n^{D}(2C_2)^{-\frac{1}{2}}L^{-1}(n)\left(\hat{c}_n-c(H)\right)\xrightarrow{D} Z_{2, H}(1)\left(4
\sum\limits_{k=1}^{2}\sum\limits_{l=1}^{2}\alpha_{l, k}\right),
\end{align*}
where $\left(\alpha_{l,k}\right)_{1\leq l,k\leq h}=\Sigma_{h+1}^{-1}C\Sigma_{h+1}^{-1}$ with 
\begin{align*}
C=\Ex\left((X_1,...,X_h)\frac{1}{\#\bar{\pi}}\left[1_{\left\{\tilde{\Pi}(X_1,...,X_h)\in\bar{\pi}\right\}}-\Px(\tilde{\Pi}(X_1,...,X_h)\in\bar{\pi})\right](X_1,...,X_h)^t\right)
\end{align*}
for $\bar{\pi}=\{(2,0,1),(1,0,2),(0,2,1),(2,0,1)\}$. Therefore, and according to Example \ref{ex: coeff}, we arrive at
\begin{align*}
\sum\limits_{k=1}^{2}\sum\limits_{l=1}^{2}\alpha_{l, k}=-\varphi^2(0)\sqrt{\frac{2-2^{2H-1}}{2^{2H-1}}}=-\varphi^2(0)\sqrt{2^{2-2H}-1}.
\end{align*}
We also know that $2C_2=2\left((1-2D)(2-D)\right)^{-1}=\left(H(4H-3)\right)^{-1}$ and since 
\begin{align*}
r_H(k)\sim H(2H-1)k^{2H-2},
\end{align*}
we get $L(n)\sim H(2H-1)$ (see \cite{beran:feng:ghosh:kulik:2013}, p. 34) with $f(k)\sim g(k)$ meaning that $\lim_{k\rightarrow\infty}\frac{f(k)}{g(k)}=1$.
All in all, it follows that
\begin{align*}
n^{2-2H}\frac{\sqrt{4H-3}}{\sqrt{H}(2H-1)}\left(   \hat{c}_n-c(H)     \right)\xrightarrow{D} Z_{2, H}(1)\left(-4
\varphi^2(0)\sqrt{2^{2-2H}-1}\right).
\end{align*}
We have $H=g(c(H))$ and $\hat{H}_n=g(\hat{c}_n)$. Due to $c(H)\in(0,\frac{2}{3})$ for $H\in (0,1)$,  $g'(c(H))=-\frac{\pi}{2\log2}\tan\left(\frac{\pi c(H)}{2}\right)$ exists and does not equal  zero for  $H\in (0,1)$. Applying Theorem 3 in \cite{vandervaart:2000} we arrive at the above limit.
\end{proof} 
\section{Proofs}
In this section the proofs of the results derived in Section 3 are presented. We are able to give these results in a more general way than needed in the context of ordinal patterns and therefore consider a larger class of functions.  
In the following we consider the asymptotic behaviour of the partial sums
\begin{equation}
  \sum_{i=1}^n \left \{ f(X_i,\ldots,X_{i+p-1}) -\Ex f(X_1,\ldots,X_{p}) \right\},\label{eq:p-sum2}
\end{equation}
where $f:\R^{p}\rightarrow \R$ satisfies $\Ex(f(X_1,\ldots,X_{p}))^2<\infty$.
A first result on the asymptotic behaviour of these partial sums, which includes the statement of Theorem \ref{th:consistent}, is given by the following proposition that can be derived from Birkhoff's ergodic theorem; see also
\cite{sinn:keller:2011}. 
\begin{proposition}\label{prop:consistency} 
Suppose that $(X_i)_{i\geq 1}$ is a stationary ergodic process, and that $f:\R^p\rightarrow \R$ is a measurable function such that $f(X_1,\ldots,X_p)\in L_1$. Then,
\[
  \frac{1}{n} \sum_{i=1}^n f(X_i,\ldots,X_{i+p-1}) \longrightarrow \Ex f(X_1,\ldots,X_p)
\]
almost surely, as $n\rightarrow \infty$.
\end{proposition}
\begin{proof}
Ergodicity of the process $(X_i)_{i\geq 1}$ means that the shift operator 
\[
  S:\R^\bbn \rightarrow \R^\bbn
\]
defined by $(x_i)_{i\geq 1}\mapsto (x_{i+1})_{i\geq 1}$, is an ergodic transformation on the sequence space $\R^\bbn$, equipped with the product $\sigma$-field and the probability measure 
$\mu={\mathcal L}((X_i)_{i\geq 1})$. Thus, by Birkhoff's ergodic theorem, we obtain for any integrable function $g:\R^\bbn\rightarrow \R$
\[
 \lim_{n\rightarrow \infty} \frac{1}{n} \sum_{k=1}^n  g(S^k(\omega))=\int g(\omega)d\mu(\omega)
\]
almost surely.  We now apply the ergodic theorem to the function $g:\R^\bbn\rightarrow \R$, defined by
\[
  g((x_i)_{i\geq 1}):= f(x_1,\ldots,x_p).
\]
With this choice of $g$, we obtain $g(S^k((x_i)_{i\geq 1}))=f(x_{k+1},\ldots,x_{k+p})$ and 
$\int g(\omega) d\mu(\omega)=\Ex g((X_i)_{i\geq 1})=\Ex f(X_1,\ldots,X_p)$, and thus
\[
 \frac{1}{n} \sum_{k=1}^n  f(x_{k+1},\ldots,x_{k+p}) \rightarrow \Ex f(X_1,\ldots,X_p)
\]
for $\mu$-almost every sequence $(x_i)_{i\geq 1}$.  Thus, by definition of $\mu$, we find
\[
 \frac{1}{n} \sum_{k=1}^n  f(X_{k+1},\ldots,X_{k+p}) \rightarrow \Ex f(X_1,\ldots,X_p),
\]
almost surely.
\end{proof}

\begin{remark}
A stationary Gaussian process $(X_k)_{k\geq 0}$ with autocovariance function 
 $r(k)=\Cov(X_0,X_k)$ such that $r(k)\xrightarrow{} 0$ if $k\xrightarrow{}\infty$ is mixing and hence ergodic; see \cite{samorodnitsky:2007}, pp. 43, 46. Thus, we may apply the above results to such Gaussian processes. 
\end{remark}

In general, in order to derive the limit distributions of partial sums given in \eqref{eq:p-sum2},  a careful analysis of the Hermite rank of the considered function $f$ is crucial. First, we present two lemmas that are helpful tools to determine the Hermite rank of a function $f$ and later on we give the proofs of Lemma \ref{lemm: estimatorHR1} and Lemma \ref{lemm: estimatorHR2}, in which we deal with the heuristic estimator of ordinal pattern probabilities and the improved estimator based on the Rao-Blackwellization, respectively. Furthermore, we give the justification for the Hermite rank of the estimator considered in Remark \ref{rem:specialcase}.\par
It is well known that for $\Sigma_{h+1}=AA^t$ we have $m(f,\Sigma_{h+1})=m(f\circ A,E_{h+1})$ and that $m(f,\Sigma_{h+1})\neq m(f,E_{h+1})$ in general; see   \cite{beran:feng:ghosh:kulik:2013}, Lemma 3.7. The last fact is disadvantageous since determining $m(f\circ A,E_{h+1})$ is usually much more complicated than determining $m(f,E_{h+1})$. 
However, it is possible to show that under a mild additional assumption $m(f\circ A,E_{h+1})$ is bounded by $m(f,E_{h+1})$; see Lemma \ref{lemma_1}.

\begin{lemma}\label{lemma:L2-convergence}
Let $(f_n)_{n\in\bbn}$ be a sequence of measurable functions in  $L^2(\N(0,E_{h+1}))$,  and let $f$ be another function in this space such that $f_n\to f$ in the metric of this space. Furthermore, let $\Sigma_{h+1}$ be a positive definite matrix such that $(\Sigma_{h+1}^{-1}-E_{h+1})$ is positive semidefinite. In this case, $(f_n)_{n\in\bbn}$ converges to $f$ in  $L^2(\N(0,\Sigma_{h+1}))$.
\end{lemma}

\begin{proof}
Let $h$ be the Radon-Nikodym density of $\N(0,\Sigma_{h+1})$ with respect to $\N(0,E_{h+1})$ such that we have
\[
  \int_{\bbr^d}\abs{f_n-f}^2 \ d\N(0,\Sigma_{h+1}) = \int_{\bbr^d} \abs{f_n-f}^2  h \ d\N(0, E_{h+1}).
\]
Hence, proving Lemma \ref{lemma:L2-convergence} boils down to the boundedness of $h$ which is obtained in the above setting by an elementary calculation:
let $\varphi$ and $\widetilde{\varphi}$ denote the density of $\N(0, E_{h+1})$ and $\N(0,\Sigma_{h+1})$, respectively. Then,
\begin{align*}
  h(x)=\frac{\widetilde{\varphi(x)}}{\varphi(x)}&= \frac{\det(E_{h+1})}{\det(\Sigma_{h+1})} \exp \left( -\frac{1}{2} (x'\Sigma_{h+1}^{-1}x - x'E_{h+1} x)\right) \\
	                                              &= \frac{1}{\det(\Sigma_{h+1})} \exp \left( -\frac{1}{2} x'(\Sigma_{h+1}^{-1}-E_{h+1})x \right). 
\end{align*}
\end{proof}

Given the above result, we arrive at the following upper bound for 
 $m(f\circ A,E_{h+1})$:
\begin{lemma}\label{lemma_1}
Let $f:\RR^p\rightarrow\RR$ be square-integrable with respect to $\N(0,E_{h+1})$ and let $\Sigma_{h+1}=AA^t$ be a $(h+1)\times (h+1)$ positive definite covariance matrix such that $(\Sigma_{h+1}^{-1}-E_{h+1})$ is positive semidefinite. Then,
\begin{equation*}
m(f\circ A,E_{h+1})\leq m(f,E_{h+1}).
\end{equation*}
\end{lemma}

\begin{remark}
Note that, for $\rho\neq 0$,
\begin{align*}
x^t\left(\rho\Sigma_{h+1}^{-1}-E_{h+1}\right)x >0 \ \Leftrightarrow  \ \frac{x^t\Sigma_{h+1}^{-1}x}{x^tx}>\frac{1}{\rho}.
\end{align*}
With  $\lambda_{\min}(\Sigma^{-1})$ denoting the smallest eigenvalue of $\Sigma^{-1}$, we have
\begin{align*}
\frac{x^t\Sigma_{h+1}^{-1}x}{x^tx}>\frac{1}{\rho}\geq \lambda_{\min}(\Sigma^{-1}).
\end{align*}
Given that $\Sigma_{h+1}$, and thus  $\Sigma_{h+1}^{-1}$, are positive definite matrices, $\lambda_{\min}(\Sigma^{-1})>0$ so that  we can choose $\rho$ such that $\rho\Sigma_{h+1}^{-1}-E_{h+1}$ is  positive semidefinite.
Since ordinal patterns are not affected by scaling, we may for this reason  assume that  $(\Sigma_{h+1}^{-1}-E_{h+1})$ is positive semidefinite.
\end{remark}

\begin{proof}
Expanding both $f$ and $f\circ A$ in Hermite polynomials with respect to $\N(0,E_{h+1})$ yields
\begin{align}
f(\U)=&\sum_{k=m_1}^{\infty}\sum_{l_1+\ldots+l_{h+1}=k}\frac{J^1_{l_1,\ldots,l_{h+1}}}{l_1!\cdots l_{h+1}!}H_{l_1,\ldots,l_{h+1}}(\U) \label{eq:proof_1}\\
(f\circ A)(\U)=&\sum_{k=m_2}^{\infty}\sum_{l_1+\ldots+l_{h+1}=k}\frac{J^2_{l_1,\ldots,l_{h+1}}}{l_1!\cdots l_{h+1}!}H_{l_1,\ldots,l_{h+1}}(\U),\notag
\intertext{where $m_1=m(f,E_{h+1})$ and $m_2=m(f,\Sigma_{h+1})$. Using Lemma \ref{lemma:L2-convergence} we can replace $\U$ by $A\cdot\U$ in \eqref{eq:proof_1} such that}
(f\circ A)(\U)=&\sum_{k=m_1}^{\infty}\sum_{l_1+\ldots+l_{h+1}=k}\frac{J^1_{l_1,\ldots,l_{h+1}}}{l_1!\cdots l_{h+1}!}(H_{l_1,\ldots,l_{h+1}}\circ A)(\U).\label{eq:proof_2}
\intertext{Each polynomial $H_{l_1,\ldots,l_{h+1}}\circ A$ can be represented by a linear combination of multivariate Hermite polynomials of degree less than or equal to $l_1+\ldots+l_{h+1}$. Therefore, we can rewrite \eqref{eq:proof_2} to}
(f\circ A)(\U)=&\sum_{k=m_3}^{\infty}\sum_{l_1+\ldots+l_{h+1}=k}\frac{J^3_{l_1,\ldots,l_{h+1}}}{l_1!\cdots l_{h+1}!}H_{l_1,\ldots,l_{h+1}}(\U),\notag
\end{align}
with $m_3\leq m_1$. By uniqueness of the Hermite decomposition we have $m_2=m_3$, which completes the proof.
\end{proof}
Proof of Lemma \ref{lemm: estimatorHR1}:
\begin{proof}
Since ordinal patterns are not affected by scaling, we may assume that $(\Sigma_{h+1}^{-1}-E_{h+1})$ is positive semidefinite. According to Lemma \ref{lemma_1} it suffices to show $E(Y_k1_{\{\tilde{\Pi}(Y_1,\ldots,Y_{h})=\pi\}})\neq 0$ for some independent standard normal random variables $Y_1,\ldots,Y_{h}$ and some $1\leq k\leq h$. For simplicity, we regard the pattern $\pi=(h,...,0)$ which corresponds to the event $\{Y_i\geq 0,~i=1,...,h\}$. Hence, we arrive at
\begin{align*}
\Ex(Y_11_{\{Y_1\geq 0,\ldots,Y_h\geq 0\}})
=&\int\limits_{0}^{\infty}\int\limits_{0}^{\infty}\ldots\int\limits_{0}^{\infty}
      y_1\varphi(y_1)\varphi(y_2)\cdots \varphi(y_{h})dy_1dy_2\cdots dy_{h}\\
=&\left(\frac{1}{2}\right)^{h-1}\varphi(0)\neq0.
\end{align*}
It follows by the same reasoning that
none of the expected values that correspond to the other ordinal patterns equals zero.
\end{proof}
Proof of Lemma \ref{lemm: estimatorHR2}:
\begin{proof}
For $\Sigma_h=AA^t$ we have $m(f,\Sigma)=m(f\circ A,E_{h})$; see \cite{beran:feng:ghosh:kulik:2013}, Lemma 3.7. According to Lemma \ref{lemma_1}, we have 
\begin{align*}
m(f\circ A,E_{h})\leq m(f,E_{h}).
\end{align*}
As a result, it is sufficient to show that $ m(f,E_{h})\leq 2$, such that 
we may conclude $m(f,\Sigma_h)=2$.

To this end, let $\Y_1=(Y_1, \ldots, Y_h )$ be a standard Gaussian random vector (i.e. with autocovariance matrix $E_h$).
Following the arguments above, we only need to consider the second order Hermite coefficients of $\hat{q}_n(\pi)$ for a fixed pattern $\pi\in\bar{\pi}$:
\begin{align*}
&b_{jk}^{\pi}=\Ex\left(Y_k Y_j 1_{\left\{\tilde{\Pi}\left(Y_1,...,Y_h\right)=\pi\right\}}\right),~ 1\leq k<j\leq h,  \text{      and}\\
&b_{jj}^{\pi}=\Ex\left(\left(Y_j^2-1\right)1_{\left\{\tilde{\Pi}\left(Y_1,...,Y_h\right)=\pi\right\}}\right),~j=1,...,h.
\end{align*}
For simplicity we regard $\pi=(h,h-1,...,0)$. Note that for this pattern it suffices to show $b_{jk}^{\pi}\neq 0$ to prove $ m(f,E_{h})\leq 2$, since in this case $c_{i,j}=c_{i,j}^{\pi}$ for $i,j=1,...,h.$
For $j\neq k$
\begin{align*}
b_{jk}^{\pi}&=\int_0^{\infty}...\int_0^{\infty} y_j y_k \varphi (y_1)...\varphi(y_{h})\mathrm{d}y_1...\mathrm{d}y_{h}\\
&=\frac{1}{2^{h-2}}\int_0^{\infty} y_j \varphi(y_j) \mathrm{d}y_j \int_0^{\infty} y_k \varphi(y_k) \mathrm{d}y_k \\
&=\frac{\varphi^2(0)}{2^{h-2}}.
\end{align*}
\end{proof}

Proof of the determination of the Hermite rank of the considered estimator in Remark \ref{rem:specialcase}:
\begin{proof}
Let $(X_k)_{k \geq 0}$ be a stationary, long-range dependent,  standard normal Gaussian process and let $h\in\NN$. 
By Lemma \ref{lemma_1} it is enough to show that $\Ex(Y_k1_{\{\Pi(Y_0,\ldots,Y_{h})=\pi\}})\neq 0$ for some independent standard normal random variables $Y_0,\ldots,Y_{h}$ and some $0\leq k\leq h$. Without loss of generality let $\pi=id$ and set $k=0$. This yields
\begin{align*}
\Ex(Y_01_{\{Y_0\leq\ldots\leq Y_{h}\}})
=&\int\limits_{-\infty}^{\infty}\int\limits_{-\infty}^{x_{h+1}}\int\limits_{-\infty}^{x_h}\ldots\int\limits_{-\infty}^{x_2}
      x_1\varphi(x_1)\varphi(x_2)\cdots \varphi(x_{h+1})dx_1dx_2\cdots dx_{h+1}\\
=&-\int\limits_{-\infty}^{\infty}\int\limits_{-\infty}^{x_{h+1}}\ldots\int\limits_{-\infty}^{x_3}
      \varphi(x_2)^2\varphi(x_3)\cdots \varphi(x_{h+1})dx_2\cdots dx_{h+1}\\
\neq&0      
\end{align*}
since we integrate a strictly positive function. Hence, for any $\pi\in S_h$ the Hermite rank of the function $f:\mathbb{R}^{h+1} \longrightarrow \mathbb{R}$, defined by
\begin{equation*}
f(x_0, x_1, \ldots, x_h):=1_{\{\Pi(x_0,\ldots,x_{h})=\pi\}}-\Px(\Pi(x_0,\ldots,x_{h})=\pi),
\end{equation*}
is equal to 1.
\end{proof}

We have finished our preparations and are now able to turn to the limit theorems for the partial sums in \eqref{eq:p-sum2} for function $f$ with Hermite rank $1$ and $2$. In the following, we will assume without loss of generality that
\[
  \Ex (f(X_1,\ldots,X_{p}))=0.
\]
We want to apply the results of \cite{arcones:1994} which hold for partial sums of functions of $\R^p$-valued random vectors $\Y_i$ that have a multivariate standard normal distribution; see also \cite{major:2019} for an alternative approach. Thus, we need to transform the vector $\X_i:=(X_i,\ldots,X_{i+p-1})$ accordingly. Let $\Sigma_{p}$ denote the covariance matrix of the vector $(X_1,\ldots,X_{p})$. Observe that $\Sigma_p$ is a Toeplitz matrix whose entries are determined by the autocovariance function 
$r(i)=\Ex (X_1 X_{1+i})$ of the process $(X_i)_{i\geq 0}$, i.e.,
\[
  \Sigma_p=\left( r(i-j)  \right)_{1\leq i,j\leq p}.
\] The Cholesky decomposition yields
\[
  \Sigma_p=AA^t,
\]
where $A$ is an upper triangular matrix.  Thus, there exists a standard normally distributed random vector $Y_i$ such that 
\[
  \X_i=A\, \Y_i.
\]

We can rewrite the partial sum in \eqref{eq:p-sum} in terms of the random vectors $\Y_i$ as follows:
\begin{align}\label{eq:f=g}
  \sum_{i=1}^n f(X_i,\ldots,X_{i+p-1}) =\sum_{i=1}^n f(A\Y_i) =\sum_{i=1}^n g(\Y_i),
\end{align}
where  $g:\R^{p}\rightarrow \R$ is defined by
\[
  g(y)=f(A y).
\]

In order to characterize the asymptotic distribution of the considered partial sum process, we apply Theorem~6 of \cite{arcones:1994}. Employing the special structure of $\Y_i$  we obtain explicit representations of the limit distributions for the cases Hermite rank equal to $1$ and  $2$.

\subsection{Limit theorems for functions with Hermite rank 1}
First we consider the asymptotic behaviour of function $f$ with Hermite rank $1$. Note that this Theorem implies the statement of Theorem \ref{th:lt-he}.

\begin{theorem}
\label{th:clt-f}
Let $(X_j)_{j\geq0}$ be a  stationary, long-range dependent standard Gaussian
process with autocovariance function
$r(k)=L(k)k^{-D}$ and let  $f:\R^{p}\rightarrow \R$ be a function with Hermite rank $1$ satisfying $\Ex(f(X_1,\ldots,X_{p}))^2<\infty$. Then,
\[
  \frac{1}{n^{1-D/2}L^{1/2}(n)} \sum_{i=1}^n \left( f(X_i,\ldots,X_{i+p-1}) 
  -\Ex f(X_1,\ldots,X_{p})\right) \longrightarrow \N \left(0,c_D \left(\sum_{j=1}^{p} \alpha_j\right)^2\right),
\]
where $c_D=\frac{2}{(1-D)(2-D)}$ and where $\alpha=\left(\alpha_1,\ldots,\alpha_{p}\right)^t= \Sigma_p^{-1} c$ with $c=\left(c_1,\ldots,c_p\right)^t=\Ex (f(X_1,\ldots,X_p)\X_1)$.
\end{theorem}

\begin{proof}
Given that the function $f$ has Hermite rank $m(f,\Sigma_p)=1$, the limit behaviour corresponds to the asymptotic behaviour of the first order term in the Hermite expansion of $f$. The Hermite rank $m(f,\Sigma_p)$ of $f$ with respect to $\X_i$ is the same as the Hermite rank $m(g,E_p)=m(f\circ A,E_p)$ of $g$ with respect to $\Y_i$; see \cite{beran:feng:ghosh:kulik:2013}, Lemma 3.7.
Since $f(X_i,\ldots,X_{i+p-1})=g(\Y_i)$ this first order term is given by
\begin{equation}
  \sum_{j=1}^{p} \left(\Ex \left(g(\Y_i) Y_{i+j-1}  \right)  \right)\,  Y_{i+j-1},
\label{eq:h-exp}
\end{equation}
with $1\leq j\leq p$, since $\Y_i=(Y_i,...,Y_{i+p-1})$. It follows by stationarity and by definition of the process $(\Y_i)_{i\geq 0}$ that the coefficient in the Hermite expansion \eqref{eq:h-exp} corresponds to
\[
b_j:= \Ex \left(  g(\Y_i) Y_{i+j-1} \right)= \Ex ( g(\Y_1) Y_{j}  )
= \Ex\left( f(X_1,\ldots,X_{1+h}) Y_{j}  \right).
\]
We can thus express the vector of coefficients $b:=\left(b_1,\ldots,b_{p}\right)^t$ as follows:
\begin{eqnarray*}
b&=& \Ex \left( f(X_1,\ldots,X_{p}) \Y_1  \right)\\
&=& \Ex \left( f(X_1,\ldots,X_{p}) A^{-1} \X_1 \right) =A^{-1}\cdot \Ex \left(f(X_1,\ldots,X_{p})  \X_1 \right)
=A^{-1}\cdot c,
\end{eqnarray*}
where $c=(c_1,\ldots,c_{p})^t$ is the vector of inner products of the random variables $X_1,\ldots, X_{p}$ with $f(X_1,\ldots,X_{p})$, i.e.,
\[
  c_k=\Ex \left(f(X_1,\ldots,X_{p}) X_k\right), \; 1\leq k\leq p.
\]
According to the results of \cite{arcones:1994}, we know that the partial sums $\sum_{i=1}^n g(Y_i)$ are dominated by the corresponding partial sums of the first order term in the Hermite expansion, i.e., that
\begin{equation}
  \sum_{i=1}^n g(\Y_i)= \sum_{i=1}^n \left( \sum_{j=1}^{p} \left( \Ex\left(g(Y_i) Y_{i+j-1}  \right)\right)  \, Y_{i+j-1}
  \right) +o_{\Px}(n^{1-D/2}L^{1/2}(n)),
\label{eq:sum-approx}
\end{equation}
where for a sequence of random variables $(X_n)_{n\in\mathbb{N}}$ we write $X_n=o_{\Px}(n)$ if $\frac{X_n}{n}\xrightarrow{\Px} 0$.

With the notations introduced above, we obtain 
\[
 \sum_{j=1}^{p} \Ex\left(g(\Y_i) Y_{i+j-1}  \right)  \cdot Y_{i+j-1} =b^t \Y_i =b^t A^{-1} \X_i =\sum_{j=1}^{p} \alpha_j X_{i+j-1},
\]
where the vector $\alpha=(\alpha_1,\ldots,\alpha_{p})^t$ is given by
\[
  \alpha:=(b^tA^{-1})^t =\left(A^{-1}\right)^t b=\left(A^{-1}\right)^t A^{-1} c = \Sigma_p^{-1} c. 
\]
Thus, we obtain 
\begin{align*}
  \sum_{i=1}^n g(\Y_i)&= \sum_{i=1}^n \left( \sum_{j=1}^{p} \alpha_j X_{i+j-1}
  \right) +o_{\Px}(n^{1-D/2}L^{1/2}(n))\\
  &=\sum_{j=1}^{p} \left\{ \alpha_j \left( \sum_{i=1}^n X_{i+j-1} \right)\right\}+o_{\Px}(n^{1-D/2}L^{1/2}(n))\\
&= \left( \sum_{j=1}^{p} \alpha_j\right) \sum_{i=1}^n X_i+o_{\Px}(n^{1-D/2}L^{1/2}(n)).
\end{align*}
The distribution of the partial sum $\sum_{i=1}^n X_i$ on the right-hand side can be  calculated exactly, as this is a partial sum of normal random variables.
\end{proof}

In the following, we study partial sums of functions of increments of a stationary long-range dependent Gaussian process of the following type 
\[
 \sum_{i=1}^n \tilde{f}(X_{i+1}-X_i,\ldots,X_{i+p-1}-X_{i+p-2}).
\]
This is a special case of partial sums of the type $\sum_{i=1}^n f(X_i,\ldots,X_{i+p-1})$, where 
\[
  f(x_1,\ldots,x_p)=\tilde{f}(x_2-x_1,\ldots,x_p-x_{p-1}).
\]
Functions of this kind appear, e.g. when studying ordinal patterns, cf. Section 3, considerations in \eqref{incrementmapping}. Therefore the following lemma combined with Theorem \ref{th:clt-f} yield the justification for the asymptotic result derived in Remark \ref{rem:specialcase}.

\begin{lemma}\label{lemma:degenerate-limit}
If $f$ can be written as a function of the increments, we have
\[
  \sum_{i=1}^p \alpha_i=0.
\]
\end{lemma}
\begin{proof}
We use a well-known fact about Gaussian random variables: Let $Y=(Y_1,\ldots,Y_p)^t$ be a vector of 
independent standard normally distributed random variables $Y_i$, $1\leq i \leq p$, and let $C_1\in \R^{k\times p}$ and $C_2\in \R^{l\times p}$ be two matrices. Then, the random vectors $C_1Y$ and $C_2Y$ are independent, if and only if each of the rows of $C_1$ is orthogonal to each of the rows of $C_2$, i.e., when
$C_1 C_2^t=0$. 

We then use the representation of $\alpha$ that we derived in the course of the proof of Theorem~\ref{th:clt-f}, namely 
\begin{eqnarray*}
  \alpha&=& \Sigma_p^{-1} \Ex\left\{ \tilde{f}((X_2-X_1,\ldots,X_p-X_{p-1})^t ) (X_1,\ldots,X_p)^t \right\} \\
  &=&  \Ex\left\{ \tilde{f}((X_2-X_1,\ldots,X_p-X_{p-1})^t) \Sigma_p^{-1} (X_1,\ldots,X_p)^t \right\} \\
  &=&  \Ex\left\{ \tilde{f}(U X) \Sigma_p^{-1}X\right\},
\end{eqnarray*}
where $X=(X_1,\ldots,X_p)^t$, and  where $U$ is the $(p-1)\times p$ matrix defined by
\[
  U=\left( 
 \begin{array}{rrrrrrrrr}
 -1 & 1 & 0 &0 &\ldots & 0 &0 &0& 0 \\
  0 & -1 &1 &0  & \ldots & 0 & 0 &0&0\\
  \vdots &  \vdots &  \vdots &  \vdots &  \vdots &  \vdots &\vdots &\vdots &\vdots  \\
  0 & 0 &0& 0& \ldots &0 & -1 &1 &0 \\
  0 & 0 &0& 0& \ldots &0 & 0 & -1 &1 \\
 \end{array}  
   \right).
\]
Let $\Sigma_p^{1/2}$ be a positive definite symmetric matrix such that $\Sigma_p^{1/2} \Sigma_p^{1/2}= \Sigma_p$ and let $\Sigma_p^{-\frac{1}{2}}$ be its inverse. Then, $Y:=\Sigma_p^{-1/2}X$ has a $p$-variate standard normal distribution. With this notation, we can rewrite the above expression for $\alpha$ as follows:
\[
  \alpha=\Ex\left\{ \tilde{f}(U X) \Sigma_p^{-1}X\right\} = \Ex\left\{ \tilde{f}(U \Sigma_p^{1/2} Y) \Sigma_p^{-1/2}Y  \right\}.
\]
If ${\mathbb I}_p^t = (1,1,\ldots,1,1)$, 
\[
  \sum_{i=1}^p\alpha_i = {\mathbb I}_p^t \alpha = \Ex\left\{ \tilde{f}(U \Sigma_p^{1/2} Y) {\mathbb I}_p^t \Sigma_p^{-1/2}Y  \right\}.
\]
Now we can apply the initial remark to the vectors $U \Sigma_p^{1/2} Y $ and ${\mathbb I}_p^t \Sigma_p^{-1/2} Y$. We have
\[
  (U\Sigma_p^{1/2})  ({\mathbb I}_p^t \Sigma_p^{-1/2})^t = U\Sigma_p^{1/2} \Sigma_p^{-1/2} {\mathbb I}_p  = U {\mathbb I}_p =0,
\]
and thus  the vectors $U \Sigma_p^{1/2} Y $ and ${\mathbb I}_p^t \Sigma_p^{-1/2} Y$ are independent. Hence,
\[
 \sum_{i=1}^p\alpha_i = {\mathbb I}_p^t \alpha = \Ex\left\{ \tilde{f}(U \Sigma_p^{1/2} Y) {\mathbb I}_p^t \Sigma_p^{-1/2}Y  \right\} 
 =\Ex\left\{ \tilde{f}(U \Sigma_p^{1/2} Y)\right\} \Ex\left\{  {\mathbb I}_p^t \Sigma_p^{-1/2}Y  \right\} =0,
\]
since $\Ex(Y)=0$.
\end{proof}

\begin{remark}
Lemma \ref{lemma:degenerate-limit} implies that the limit in Theorem 2.3 is trivial if the function $f$ can be considered as a function of the increment process of a stationary, long-range dependent Gaussian
process $(X_j)_{j\geq0}$.
An explanation for this phenomenon results from the observation that the  increments of long-range dependent time series do not display characteristic features of long-range dependence.
To see this, 
let $g$ denote the spectral density of the time series $(X_j)_{j\geq 0}$, i.e. $g$ is a non-negative function satisfying
\begin{align*}
r(k):=\Cov(X_0, X_k)=\int_{-\pi}^{\pi}e^{ik\lambda}g(\lambda)d\lambda.
\end{align*}
By assumption, we have 
\begin{align*}
r(k)=k^{-D}L(k),
\end{align*}
where $L$ is a  function that is slowly varying at infinity.
If, additionally,  $L$ is quasi-monotone,  
it follows that
\begin{align*}
g(\lambda)= |\lambda|^{D-1}L_g(\lambda)
\end{align*}
for some function $L_g$ that is  slowly varying  at zero; see \cite{pipiras:taqqu:2017}, p. 19.
For the increment process $(Z_j)_{j\geq 1}$ defined by $Z_j:=X_j-X_{j-1}$, it then holds that
\begin{align*}
\Cov(Z_1, Z_{k+1})=\int_{-\pi}^{\pi}e^{ik\lambda}2(1-\cos(\lambda))f(\lambda)d\lambda.
\end{align*}
For this reason, $\tilde{g}(\lambda):=2(1-\cos(\lambda))g(\lambda)$ corresponds to the spectral density of the process $(Z_j)_{j\geq 1}$. 
Note that
\begin{align*}
\tilde{g}(\lambda)=|\lambda|^{D+1}L_{\tilde{g}}(\lambda) 
\end{align*}
with 
\begin{align*}
L_{\tilde{g}}(\lambda):=L_{g}(\lambda)\frac{2(1-\cos(\lambda))}{\lambda^2}
\end{align*}
 slowly varying at zero.
It follows that
\begin{align*}
\sum\limits_{k=-\infty}^{\infty}r(k)=0,
\end{align*}
i.e., the increment process is antipersistent  and, in particular, short-range dependent; see \cite{pipiras:taqqu:2017},  p. 31.
\end{remark}
This finding coincides with results on limit theorems for discretely observed processes based on fractional Brownian motion, where the application of linear difference filters leads to a smaller exponent in the autocovariance function, cf. \cite{coeurjolly:2001}, \cite{istas:1997}. In our setting this would mean that considering differences of the stationary, long-range dependent process would lead to a short-range dependent process and hence to a Gaussian central limit theorem with a different normalizing constant, namely $n^{-1/2}$.
\subsection{Limit theorems for functions with Hermite rank 2}
We continue to study the asymptotic behaviour of the partial sums in \eqref{eq:p-sum} for a function $f$ with Hermite rank $2$ and therefore obtain the statement of Theorem \ref{th:lt-ie} along the way.
\begin{theorem}\label{th:nclt-f}
Let $(X_j)_{j\geq1}$ be a  stationary, long-range dependent standard Gaussian
process with  autocovariance function
$r(k)=L(k)k^{-D},~~k\geq 1$, and let  $f:\R^{p}\rightarrow \R$ be a function with Hermite rank $2$ satisfying $\Ex(f(X_1,\ldots,X_{p}))^2<\infty$. Then,
 if $D\in (0, \frac{1}{2})$,
\[
n^{D-1}(2!C_2)^{-\frac{1}{2}}L^{-1}(n) \sum_{i=1}^n \left( f(X_i,\ldots,X_{i+p-1}) 
  -\Ex f(X_1,\ldots,X_{p})\right) \overset{\mathcal{D}}{\longrightarrow}Z_{2, (1-D/2)}(1)
\sum\limits_{k=1}^{p}\sum\limits_{l=1}^{p}\alpha_{l, k}
\]
with $(\alpha_{l,k})_{1\leq l, k\leq p}:=\Sigma_{p}^{-1} C \Sigma_{p}^{-1}$, $C=(c_{l,k})_{1\leq l, k\leq p}=\Ex \left(\X_1(f(\X_1)-\Ex f(\X_1))\X_1^t\right)$ \\
and $C_2=\left((1-2D)(2-D)\right)^{-1}$.
\end{theorem}
\noindent
\begin{remark}
The extension of the theorem above for  $D>\frac{1}{2}$ is given in \cite{arcones:1994}, Theorem 4.
\end{remark}
\begin{proof}
Recall that 
$ \X_i=A\, \Y_i$
for an upper triangular matrix $A$ with   $AA^t=\Sigma_p$ and $\Sigma_p=\left( r(i-j)  \right)_{1\leq i,j\leq p}$.
and a multivariate standard normally distributed vector $\Y_i$.
Since the Hermite rank of $g$, defined by  $g(y):=f(Ay)$, equals $2$, the partial sums $\sum_{i=1}^n g(\Y_i)$ are dominated by the corresponding partial sums of the second  order term in the Hermite expansion, i.e., that
\begin{equation*}
  \sum_{i=1}^n g(\Y_i) = \sum_{l_1+l_2+\cdots+l_{p}=2} \left(\Ex \left(g(\Y_i) H_{l_1, l_2,  \ldots, l_{p}}(\Y_i)  \right)  \right)\,  H_{l_1, l_2, \ldots, l_{p}}(\Y_i)+o_{\Px}(n^{1-D}L(n));
\end{equation*}
see  Theorem~6 in \cite{arcones:1994}.
Note that
\begin{align*}
&\sum_{l_1+l_2+\cdots+l_{p}=2} \left(\Ex \left(g(\Y_i) H_{l_1, l_2,  \ldots, l_{p}}(\Y_i)  \right) \right)\,  H_{l_1, l_2, \ldots, l_{p}}(\Y_i)\\
=&\sum\limits_{j=1}^{p}\left(\Ex \left(g(\Y_i) (Y_{i+j-1}^2-1) \right) \right)\,  (Y_{i+j-1}^2-1)
+  \sum\limits_{1\leq j, k\leq p, j\neq k}\left(\Ex \left(g(\Y_i) Y_{i+j-1}Y_{i+k-1} \right)  \right)\, Y_{i+j-1}Y_{i+k-1}\\
=&\sum\limits_{j=1}^{p}\left(\Ex \left(g(\Y_i) Y_{i+j-1}^2 \right)  \right)\,  (Y_{i+j-1}^2-1)
+ \sum\limits_{1\leq j, k\leq p, j\neq k}\left(\Ex \left(g(\Y_i) Y_{i+j-1}Y_{i+k-1} \right)  \right)\, Y_{i+j-1}Y_{i+k-1}\\
=&\sum\limits_{1\leq j, k\leq p}\Ex \left(g(\Y_i) Y_{i+j-1}Y_{i+k-1} \right)  \,  \, Y_{i+j-1}Y_{i+k-1}-\sum\limits_{j=1}^{p}\Ex \left(g(\Y_i) Y_{i+j-1}^2 \right).
\end{align*}
Since the left-hand side of the above equality is centered to mean zero, 
\begin{align*}
\Ex\left(\sum\limits_{1\leq j, k\leq p}\Ex \left(g(\Y_i) Y_{i+j-1}Y_{i+k-1} \right) \,  \, Y_{i+j-1}Y_{i+k-1}\right)=\sum\limits_{j=1}^{p}\Ex \left(g(\Y_i) Y_{i+j-1}^2 \right).
\end{align*}

With $B=(b_{j,k})_{1\leq j, k\leq p}$, where $b_{j, k}=\Ex \left(g(\Y_i) Y_{i+j-1}Y_{i+k-1} \right) =\Ex \left(g(\Y_1) Y_{j}Y_{k} \right) $, it follows that
\begin{align*}
\sum\limits_{1\leq j, k\leq p}\Ex \left(g(\Y_i) Y_{i+j-1}Y_{i+k-1} \right)   Y_{i+j-1}Y_{i+k-1}=\Y_i^tB\Y_i=\X_i^t\left(A^{-1}\right)^tBA^{-1}\X_i,
\end{align*}
where $Y_{i+j-1}$ denotes the $j$-th entry of the vector $\Y_i$, $1\leq j\leq p$, since $\Y_i=(Y_i,\ldots,Y_{i+p-1})$.

Note that $B=\Ex \left(\Y_ig(\Y_i)\Y_i^t\right)=\Ex \left(A^{-1}\X_ig(\Y_i)(A^{-1}\X_i)^t\right)=A^{-1}\Ex \left(\X_i f(\X_i)\X_i^t\right)\left(A^{-1}\right)^t$.
As a result,
\begin{align*}
\X_i^t\left(A^{-1}\right)^tBA^{-1}\X_i&=\X_i^t\left(AA^t\right)^{-1}\Ex \left(\X_i f(\X_i)\X_i^t\right)\left(AA^t\right)^{-1}\X_i\\
&=\X_i^t\Sigma_{p}^{-1}\Ex \left(\X_i f(\X_i)\X_i^t\right)\Sigma_{p}^{-1}\X_i.
\end{align*}

With $\mathcal{A}:=(\alpha_{jk})_{1\leq j,k\leq p}:=\Sigma_{p}^{-1}\Ex \left(\X_i f(\X_i)\X_i^t\right)\Sigma_{p}^{-1}$ it follows that
\begin{align*}
\X_i^t\left(A^{-1}\right)^tBA^{-1}\X_i=\sum\limits_{1\leq j, k\leq p}X_{i+j-1}X_{i+k-1}\alpha_{jk}.
\end{align*}
All in all, we arrive at
\begin{align*}
\sum\limits_{i=1}^n \sum\limits_{1\leq j, k\leq p}\Ex \left(g(\Y_i) Y_{i+j-1}Y_{i+k-1} \right)   Y_{i+j-1}Y_{i+k-1}&=
\sum\limits_{i=1}^n\sum\limits_{1\leq j, k\leq p}X_{i+j-1}X_{i+k-1}\alpha_{jk} \\
&=\sum\limits_{1\leq j, k\leq p}\sum\limits_{i=1}^n X_{i+j-1}X_{i+k-1}\alpha_{jk}.
\end{align*}

Note that
\begin{align*}
\sum\limits_{1\leq j, k\leq p}\sum\limits_{i=1}^n X_{i+j-1}X_{i+k-1}\alpha_{jk}
&=\sum\limits_{1\leq j, k\leq p}\sum\limits_{i=j}^{n+j-1} X_{i}X_{i+k-j}\alpha_{jk}\\
&=\sum\limits_{k=2}^{p}\sum\limits_{l=1}^{k-1}\alpha_{k-l, k}\sum\limits_{i=k-l}^{n+k-l-1} X_{i}X_{i+l}
+\sum\limits_{k=1}^{p}\sum\limits_{l=k-p}^0 \alpha_{k-l, k}\sum\limits_{i=k-l}^{n+k-l-1} X_{i}X_{i+l}.
\end{align*}
Define the sample covariance at lag $l$ by
\begin{align*}
\hat{r}_n(l):=\frac{1}{n}\sum\limits_{i=0}^{n-l}X_iX_{i+l}.
\end{align*}
Considering both summands separately, we arrive at
\begin{align*}
 \sum\limits_{l=1}^{k-1}\alpha_{k-l, k}\sum\limits_{i=k-l}^{n+k-l-1} X_{i}X_{i+l}
= \sum\limits_{l=1}^{k-1}\alpha_{k-l, k}n\left(\frac{1}{n}\sum\limits_{i=n-l+1}^{n+k-l-1} X_{i}X_{i+l}+\hat{r}_n(l)-\frac{1}{n}\sum\limits_{i=0}^{k-l-1} X_{i}X_{i+l}\right)
\end{align*}
and
\begin{align*}
\sum\limits_{l=k-p}^{0}\alpha_{k-l, k}\sum\limits_{i=k-l}^{n+k-l-1} X_{i}X_{i+l}
&=\sum\limits_{l=k-p}^{0}\alpha_{k-l, k}\sum\limits_{i=k}^{n+k-1} X_{i-l}X_{i}\\
&=\sum\limits_{l=0}^{p-k}\alpha_{k+l, k}\sum\limits_{i=k}^{n+k-1} X_{i+l}X_{i}\\
&=\sum\limits_{l=0}^{p-k}\alpha_{k+l, k}n\left(\frac{1}{n}\sum\limits_{i=n-l+1}^{n+k-1} X_{i+l}X_{i}+\hat{r}_n(l)-\frac{1}{n}\sum\limits_{i=0}^{k-1} X_{i+l}X_{i}\right).
\end{align*}

All in all, it follows that
\begin{align*}
&n^{D-1}(2!)^{-\frac{1}{2}}L^{-1}(n)\sum\limits_{i=1}^n \left(\sum\limits_{1\leq j, k\leq p}\Ex \left(g(\Y_i) Y_{i+j-1}Y_{i+k-1} \right)  \,  \, Y_{i+j-1}Y_{i+k-1}-\sum\limits_{j=1}^{p}\Ex \left(g(\Y_i) Y_{i+j-1}^2 \right)\right)\\
=&n^{D-1}(2!)^{-\frac{1}{2}}L^{-1}(n)\sum\limits_{1\leq j, k\leq p}\alpha_{jk}\sum\limits_{i=1}^n \left(X_{i+j-1}X_{i+k-1}-\Ex\left(X_{i+j-1}X_{i+k-1}\right)\right)\\
=&n^{D}(2!)^{-\frac{1}{2}}L^{-1}(n)
\sum\limits_{k=2}^{p}\sum\limits_{l=1}^{k-1}\alpha_{k-l, k}\left(\hat{r}_n(l)-r(l)\right)\\
&+n^{D}(2!)^{-\frac{1}{2}}L^{-1}(n)
\sum\limits_{k=1}^{p}\sum\limits_{l=0}^{p-k}\alpha_{k+l, k}\left(\hat{r}_n(l)-r(l)\right)+o_{\Px}(1).
\end{align*}
If $D\in (0, \frac{1}{2})$, it follows by Section 4.4.1.3 in \cite{beran:feng:ghosh:kulik:2013} that
\begin{align*}
n^{D}(2!C_2)^{-\frac{1}{2}}L^{-1}(n)\left(\hat{r}_n(0)-r(0), \ldots, \hat{r}_n(p)-r(p)\right)\overset{\mathcal{D}}{\longrightarrow}(Z_{2, H}(1), \ldots, Z_{2, H}(1)),
\end{align*}
where $Z_{2, H}(\cdot)$ is a Rosenblatt process with parameter $H=1-\frac{D}{2}$ and $C_2=\left((1-2D)(2-D)\right)^{-1}$.

Therefore, the considered expression converges in distribution to
\begin{align*}
Z_{2, H}(1)\left(
\sum\limits_{k=2}^{p}\sum\limits_{l=1}^{k-1}\alpha_{k-l, k}+
\sum\limits_{k=1}^{p}\sum\limits_{l=0}^{p-k}\alpha_{k+l, k}\right)=Z_{2, H}(1)
\sum\limits_{k=1}^{p}\sum\limits_{l=1}^{p}\alpha_{l, k}.
\end{align*}

\end{proof}
Hence, we are able to characterize the limit distribution of the partial sums in \eqref{eq:p-sum2} for functions $f$ with Hermite rank $1$ and $2$.
\section{Simulation study}
\noindent
We simulate $N=10000$ paths of fractional Gaussian noise (by the command ``simFGN0" from the RPackage ``longmemo") with sample size $n=1000000$ for different values of $H$ to compare the distribution of the estimators $\hat{q}_{n}(\pi)$, $\hat{p}_n(\pi)$ and $\hat{H}_n$ with the theoretical results derived above. We standardized the estimators following the normalization constants given in Theorem \ref{th:lt-he} and Theorem \ref{th:lt-ie}. The results depending on the long-range dependence parameter $H$ are displayed in Figure 3 and in Figure 4.\newpage
\begin{figure}[h]
\begin{center}
\includegraphics[width=0.8\textwidth]{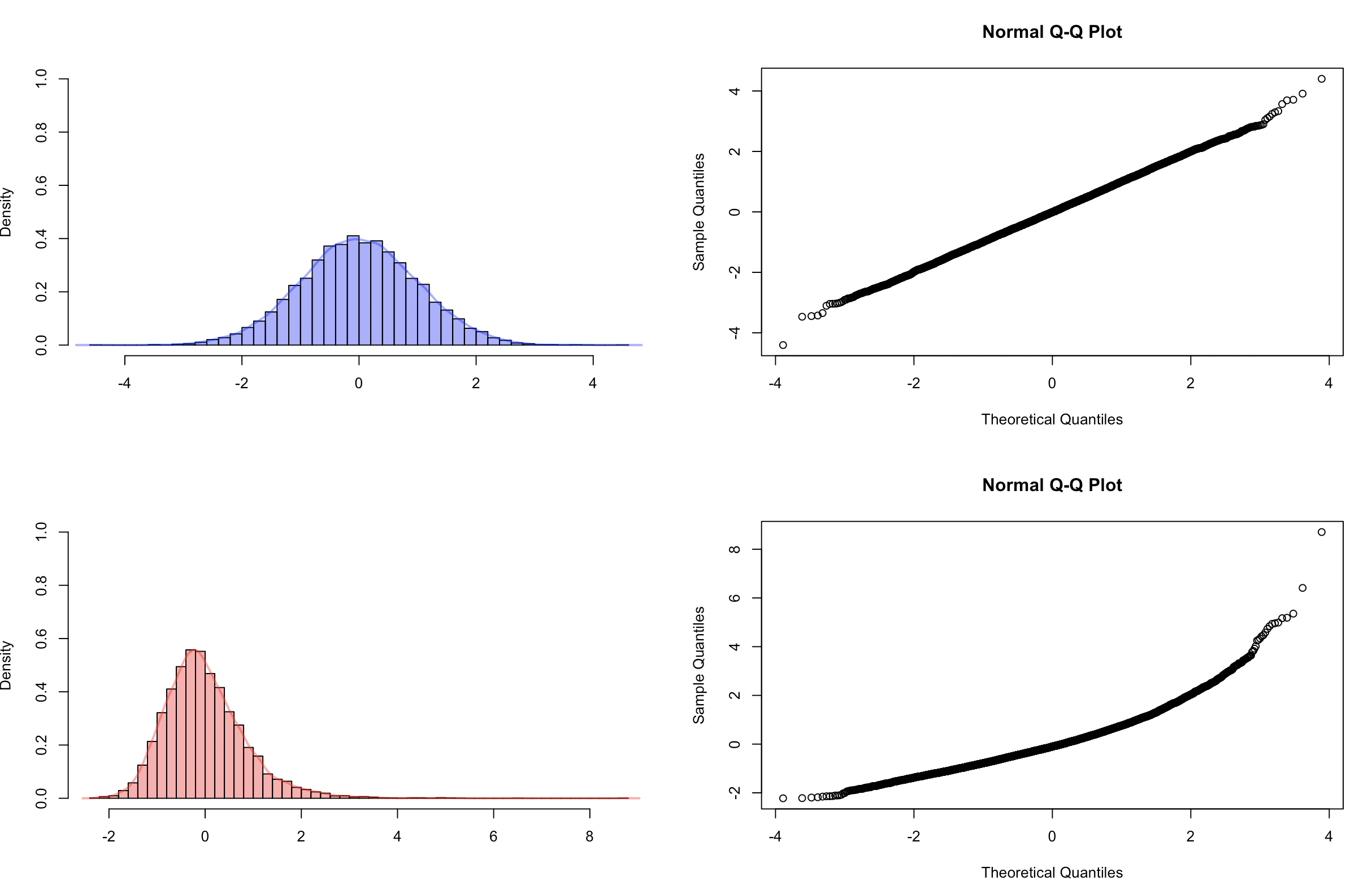}
\caption{Histogram, kernel density estimation and qqplot of the estimators $\hat{q}_n(\pi)$ (blue) and $\hat{p}_n(\pi)$ (red) for $n=1000000$ and $\pi=(2,1,0)$ in the case $H=0.8$ $(D=0.4)$.}
\end{center}
\end{figure}~\newline
\begin{figure}[h]
\begin{center}
\includegraphics[width=0.8\textwidth]{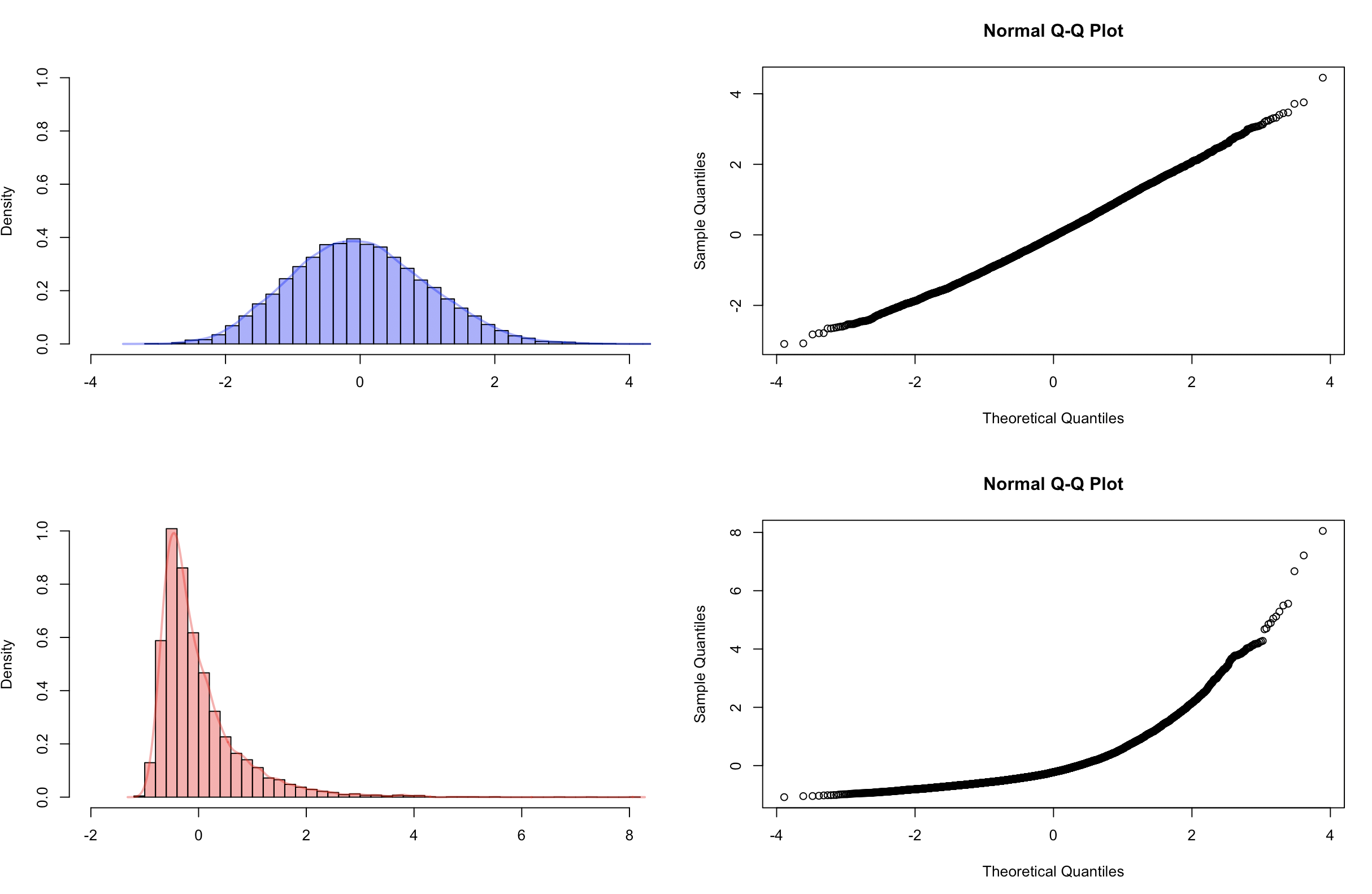}
\caption{Histogram, kernel density estimation and qqplot of the estimators $\hat{q}_n(\pi)$ (blue) and $\hat{p}_n(\pi)$ (red) for $n=1000000$ and $\pi=(2,1,0)$ in the case $H=0.9$ $(D=0.2)$.}
\end{center}
\end{figure}\newpage
In Figure 5,  the histograms and kernel density estimations of the estimator of the Hurst parameter are given, standardized by the normalizing constants we derived in Theorem \ref{th:lt-hurst}.
\begin{figure}[h]\label{fig: figure 5}
\begin{center}
\includegraphics[width=1\textwidth]{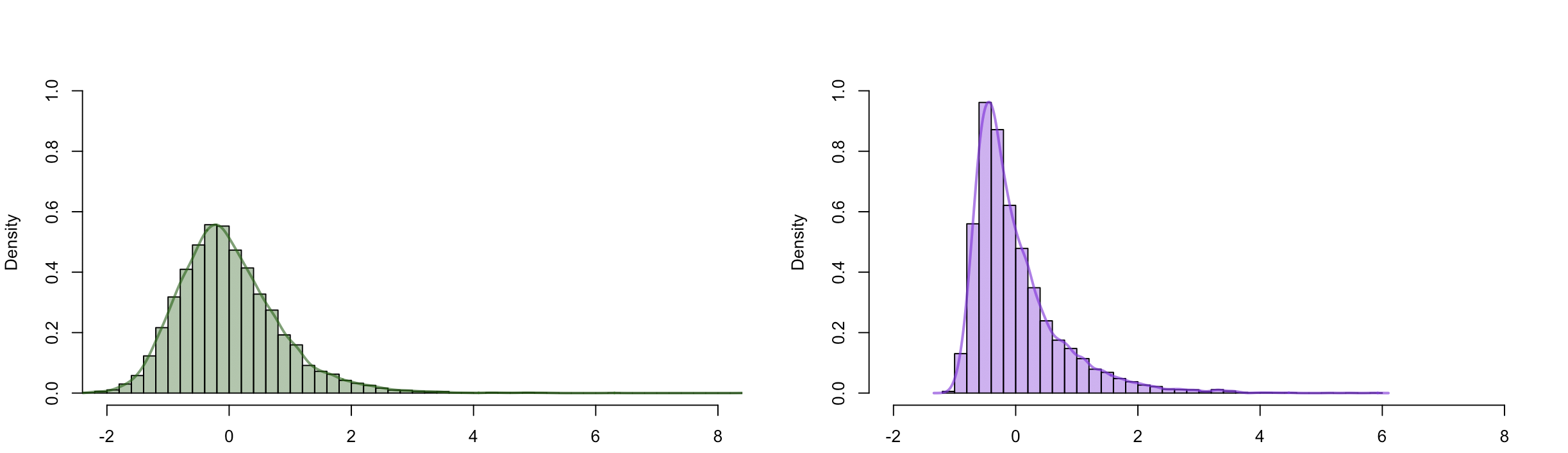}
\caption{Histogram and density of the Hurst parameter estimator for  $H=0.8$ (green) and $H=0.9$ (violet).}
\end{center}
\end{figure}\newline
Acknowledgments: We would like to thank two anonymous referees for their helpful comments.
\bibliography{bibliography_ordinal_patterns}~\newpage
\section{Appendix}
Calculation of the Hermite coefficients of $\hat{q}_n(\pi)$ for $h=2$ for the pattern $\pi=(2,1,0)$; cf. Example \ref{ex: forcholesky}.\par
Since we look at $h=2$, the covariance matrix of $\X_1=(X_1,X_2)^t$ is given by
\begin{align*}
\Sigma_2=\left( 
 \begin{array}{rr}
 1 & r(1)  \\
  r(1) & 1
 \end{array}  
   \right).
\end{align*}
The Cholesky decomposition  $\Sigma=AA^t$ has the following form:
\begin{align*}
A=\left( 
 \begin{array}{rr}
 1 &0  \\
  r(1) & \sqrt{1-(r(1))^2}
 \end{array}  
   \right)
\end{align*}
Note that $\X_1=A\Y_1$, where $\Y_1=(Y_1,Y_2)^t$ has a bivariate standard normal distribution.
Following Theorem \ref{th:lt-he}, we need to calculate $\alpha=\left(A^{-1}\right)^t b$, where $b=\mathbb{E}\left(\Y_1 1_{\left\{\tilde{\Pi}(X_1,X_2)=\pi\right\}}\right)$. Since
\begin{align*}
\left(A^{-1}\right)^t=\left( 
 \begin{array}{rr}
 1 & -\frac{r(1)}{\sqrt{1-(r(1))^2}}  \\
  0 & \frac{1}{\sqrt{1-(r(1))^2}},
 \end{array}  
   \right)
\end{align*}
we need to determine $b$ to calculate the variance in the limit distribution.

We consider $\pi=(2,1,0)$. From the Cholesky decomposition it follows that $X_1=Y_1$ and $X_2=r(1)Y_1+\sqrt{1-(r(1))^2}Y_2$ and therefore $c_1=\mathbb{E}\left(X_1 1_{\left\{\tilde{\Pi}(X_1,X_2)=\pi\right\}}\right)=b_1$ and $c_2=r(1)b_1+\sqrt{1-(r(1))^2}b_2$. For this choice of $\pi$ we also know by \eqref{eq:spacereversion} and \eqref{eq:timereversion} that $c_1=c_2$ and hence we arrive at
\begin{align*}
b_1=\frac{\sqrt{1-(r(1))^2}}{1-r(1)}b_2.
\end{align*}
Therefore, it is sufficient to only determine $b_2$. For this, we rewrite 
\begin{align*}
\{\tilde{\Pi}(X_1,X_2)=(2,1,0)\}&=\{X_1\geq 0,X_2\geq 0\}=\{Y_1\geq 0,r(1)Y_1+\sqrt{1-(r(1))^2}Y_2\geq 0\}\\
&=\{Y_1\geq 0,Y_2\geq -\frac{r(1)}{\sqrt{1-(r(1))^2}}Y_1\}.
\end{align*}
Hence, we need to determine
\begin{align*}
b_2&=\mathbb{E}\left(Y_2 1_{\left\{\tilde{\Pi}(X_1,X_2)=\pi\right\}}\right)\\
&=\int_0^{\infty}\int_{ -\frac{r(1)}{\sqrt{1-(r(1))^2}}Y_1}^{\infty} y_2 \varphi(y_2)\varphi(y_1)\mathrm{d}y_2\mathrm{d}y_1\\
&=\int_0^{\infty}\varphi \left( \frac{r(1)}{\sqrt{1-(r(1))^2}}y_1\right)\varphi(y_1)\mathrm{d}y_1\\
&=\frac{1}{2\pi}\int_0^{\infty}\exp\left(-\frac{\left(1+\frac{(r(1))^2}{1-(r(1))^2}\right)y_1^2}{2}\right)\mathrm{d}y_1\\
&=\frac{1}{2\pi}\int_0^{\infty}\exp\left(-\frac{\left(\frac{1}{1-(r(1))^2}\right)y_1^2}{2}\right)\mathrm{d}y_1\\
&=\frac{1}{2\sqrt{2\pi}}\sqrt{1-(r(1))^2}.
\end{align*}
Finally, we obtain
\begin{align*}
\sum_{j=1}^2 \alpha_j&=b_1+\frac{1-r(1)}{\sqrt{1-(r(1))^2}}b_2\\
&=\frac{\sqrt{1-(r(1))^2}}{1-r(1)}b_2+\frac{1-r(1)}{\sqrt{1-(r(1))^2}}b_2\\
&=\left(\sqrt{\frac{1+r(1)}{1-r(1)}}+\sqrt{\frac{1-r(1)}{1+r(1)}}\right)b_2\\
&=\frac{1}{2\sqrt{2\pi}}\frac{2}{\sqrt{1-(r(1))^2}}\sqrt{1-(r(1))^2}\\
&=\frac{1}{\sqrt{2\pi}}.
\end{align*}
As a result, we confirm the result from Example \ref{ex: forcholesky} for the pattern $\pi=(2,1,0)$. For $\pi=(2,0,1)$, the analytical calculations  work analogously.
\end{document}